\documentclass[11pt,a4paper]{article}
\usepackage[utf8]{inputenc}
\usepackage[T1]{fontenc}
\usepackage{amsmath}
\usepackage{amsfonts}
\usepackage{amssymb}
\usepackage{amsthm}
\usepackage{bbm}
\usepackage[left=2cm,right=1cm,top=2cm,bottom=2cm]{geometry}

\newtheorem{thm}{THEOREM}[section]

\newtheorem{lem}[thm]{Lemma}
\newtheorem{definition}[thm]{DEFINITION}
\newtheorem{example}[thm]{EXAMPLE}
\newtheorem{proposition}[thm]{PROPOSITION}
\newtheorem{remark}[thm]{REMARK}

\providecommand{\norm}[1]{\left\lVert#1\right\rVert}
\providecommand{\abs}[1]{\left\lvert#1\right\rvert}
\providecommand{\pr}[1]{\left(#1\right)} 
\providecommand{\pp}[1]{\left[#1\right]} 
\providecommand{\set}[1]{\left\lbrace#1\right\rbrace} 
\providecommand{\scal}[1]{\left\langle#1\right\rangle}

\providecommand{\keywords}[1]{\textbf{\textit{Keywords:  }} #1}
\providecommand{\classification}[1]{\textbf{\textit{MSC:  }} #1}

\begin{document}
\author{Dan GOREAC\footnote{Université Paris-Est, LAMA (UMR 8050), UPEMLV, UPEC, CNRS, F-77454, Marne-la-Vallée, France, Corresponding author, email:  dan.goreac@u-pem.fr} }
\title{Border Avoidance: Necessary Regularity for Coefficients and Viscosity Approach}
\maketitle
\abstract{Motivated by the result of invariance of regular-boundary open sets in \cite{CannarsaDaPratoFrankowska2009} and multi-stability issues in gene networks, our paper focuses on three closely related aims. First, we give a necessary local Lipschitz-like condition in order to expect invariance of open sets (for deterministic systems). Comments on optimality are provided via examples. Second, we provide a border avoidance (near-viability) counterpart of \cite{CannarsaDaPratoFrankowska2009} for controlled Brownian diffusions and piecewise deterministic switched Markov processes (PDsMP). We equally discuss to which extent Lipschitz-continuity of the driving coefficients is needed. Finally, by applying the theoretical result on PDsMP to Hasty's model of bacteriophage (\cite{hasty_pradines_dolnik_collins_00}, \cite{crudu_debussche_radulescu_09}), we show the necessity of explicit modeling for the environmental cue triggering lysis.}\\\\
\\
\classification{93E20; 49L25; 60J60; 60J75; 34H05; 92C42}
\section{Introduction}
A special feature of controlled systems (either deterministic or presenting some randomness) is the ability to keep trajectories within confined domains $K$ (subsets of the state space). This property is known as \textit{invariance} when achieved with arbitrary control and \textit{viability} when special construction of the control is required. Since the pioneer work by Nagumo \cite{Nagumo_42}, both properties (invariance and viability) have been extensively studied in either deterministic or stochastic settings (Brownian or jump diffusion, piecewise deterministic Markov setting, etc.). The ensuing related literature is huge and we will only be able to mention a few of the available results (and methods). Traditionally, for closed sets $K$, the problem can be tackled either by contingent or quasi-tangency methods (e.g. \cite{aubin_91}, \cite{AubinDaPrato_90}, \cite{aubin_frankowska_90}, \cite{Gautier_Thibault_93}, \cite{CarjaNeculaVrabie2007}, etc.) or by using the theory of viscosity solutions (e.g. in \cite{BPQR1998}, \cite{BardiGoatin}, \cite{BardiJensen2002},  \cite{Peng_Zhu_08}, \cite{G8}, \cite{Rascanu_Nie_2012}). In \cite{CannarsaDaPratoFrankowska2009}, the authors consider the invariance problem for Brownian diffusions for an open set $\mathring{K}$ when $K$ has smooth border. They show that, whenever $K$ is a piecewise $C^{2,1}$-smooth domain, invariance of $K$ and invariance of $\mathring{K}$ are equivalent. In other terms, a solution starting from $x\in\mathring{K}$ not only remains in $K$ but it also avoids border $\partial K$. The method employed in \cite{CannarsaDaPratoFrankowska2009} relies on the regularity of the signed distance function ($\delta_K$ introduced in \cite{DelfourZolesio1994}) and a-priori estimates requiring Lipschitz-regularity of the driving coefficients. This direct method based on the use of a test function of type $-\ln\delta_K$ has also been used by the authors of \cite{BFQ2017}.\\
For complex systems of interaction between several players or species, the notion of multi-stability plays a central part. We have in mind the simplest example relevant to this framework in the hybrid modeling of gene networks: the phage lambda. Roughly speaking, in order to survive, the bacteriophage relies on the host (E-Coli). Two issues are then possible (for a temperate form of phage): either a cohabitation (lysogeny resulting in moderate replication of prophage) or high speed replication ultimately leading to excision (lysis). Following the simplified model introduced by \cite{hasty_pradines_dolnik_collins_00}, a piecewise deterministic model has been introduced in \cite{crudu_debussche_radulescu_09} (and the averaging probabilistic techniques rigorously justified in \cite{crudu_Debussche_Muller_Radulescu_2012}). In such models, migrating from lysogenic to lytic cycle   is indicated by a cro-repressor-related variable reaching some threshold level. In other words, as the model trajectory reaches this threshold level (the boundary of lysogenic region $\partial K^{lysogeny}$), the dynamics are altered to induce stable attraction to the lytic $0$-concentration of cro-repressor.\\
The aim of the present paper is threefold.\\
 First, we try to investigate to which extent the Lipschitz-regularity of the driving coefficients for the dynamics can be weakened in order to still envisage invariance of $\mathring{K}$ for compact domains $K$ of class $C^{2,1}$. By reasoning with respect to deterministic (uncontrolled) systems \begin{equation*}
dX_t^x=b\pr{X_t^x}dt,\textnormal{ for all }t\geq 0;\ X_0^x=x\in \mathbb{R}^N,
\end{equation*} we present a dychotomy-type necessary condition. This condition is motivated by the fact that what matters for invariance is not $b$ itself (see examples in Section \ref{SectionNecessary}) but the behavior locally around the border $\partial K$ of the outward velocity \[b^+(x)={\scal{b(x),\nu_K\pr{\pi_{\partial K}(x)}}}^+=\max\set{\scal{b(x),\nu_K\pr{\pi_{\partial K}(x)}},0},\]where $\pi_{\partial K}$ is the projection onto $\partial K$ and $\nu_K$ designates the outward unit normal to $K$. When the driving coefficient has a bad continuity modulus (the typical example being of Hölder type and the quality being measured by some distribution's inverse function being no worse that $\frac{1}{\norm{x-y}\log {\frac{1}{\norm{x-y}}}}$), the necessary condition for invariance of $\mathring{K}$ (in Proposition \ref{PropMainNec}) is of local Lipschitz-type \[\liminf_{x\in \mathring{K}, \delta_K(x)\rightarrow 0}\frac{\abs{b^+(x)-b^+\pr{\pi_{\partial K}(x)}}}{\norm{x-\pi_{\partial K}(x)}}<\infty.\] We present further examples concerning the optimality of the condition. In the necessary condition one cannot (in general) replace $\liminf$ with $\limsup$ (cf. Example \ref{ExpLiminfLimSup}). This condition may fail to hold for convenient continuity moduli (cf. Example \ref{Exp_xlnx}). Finally, the condition cannot be strengthened to a pointwise one (i.e. by replacing $\delta_K(x)\rightarrow 0$ with $x\rightarrow\bar{x}$ for all $\bar{x}\in \partial K$).\\\\
The second aim of the paper (Section \ref{Section_NearViabDiff}) is to provide the near-viability counterpart of \cite{CannarsaDaPratoFrankowska2009} for controlled Brownian diffusions
\begin{equation*}
\left\lbrace 
\begin{split} 
dX_t^{x,u}&=b\pr{X_t^{x,u},u(t)}dt+\sigma\pr{X_t^{x,u},u(t)}dW_t,\textnormal{ for }t\geq0,\\
X_0^{x,u}&=x\in \mathbb{R}^N.
\end{split}
\right.
\end{equation*}driven by regular coefficients. We adopt a slightly different method than the one used in \cite{CannarsaDaPratoFrankowska2009}. We note that near-viability enforced with initial data close to $\partial K$ is sufficient to imply near-viability of $\mathring{K}$. The near-viability indicator value function \[V(x):=\inf_u\mathbb{E}\pp{\int_{\mathbb{R}_+}e^{-\lambda t}\mathbbm{1}_{\pr{\mathring{K}}^c}\pr{X_t^{x,u}dt}}\]
is then approximated by regular subsolutions $\pr{V_n^{\frac{1}{n^2}}}$ of (approximating) Hamilton-Jacobi-Bellman equations. By multiplying these functions with $\delta_K$, one gets a family of smooth functions $W_n:=-\delta_K\times V_n^{\frac{1}{n^2}}$. Finally, by "comparing" (at global minimum points) these functions with the 0-constant (subsolution) and passing to the limit as $n\rightarrow\infty$,  it follows that, as soon as the closed set $K$ is near-viable, $V$ cannot be strictly positive at any point $x\in\mathring{K}$. This comparison uses extensively the (closed-set) viability characterization in \cite{BCQ2001} and some estimates for the infinitesimal operator $\mathcal{L}^u$ applied to the signed distance $\delta_K$ locally close to $\partial K$. Our results are proven for bounded Lipschitz-continuous coefficients $b$ and $\sigma$. However, Lipschitz-regularity is only required to provide the smooth approximations $V_n^{\frac{1}{n^2}}$ (using Krylov's shaking of coefficients method cf. \cite{Krylov_00}). Alternatively, generalization of Gronwall's inequality may be envisaged for suitable continuity moduli (see Remark \ref{RemLipschitzOsgood}) and the remaining steps in the proof of the main result (Theorem \ref{ThmMainDiff}) apply with local Lipschitz-like condition (similar to the necessary ones).\\\\
Third, we show that the same program (explained for Brownian diffusions) works in the setting of controlled piecewise continuous switched Markov processes. This is a particular class of piecewise deterministic Markov processes (abreviated PDMP and introduced in \cite{Davis_84},\cite{davis_93}) consisting of a couple mode/average behavior $\pr{\Gamma, X}$. The switch changes the mode and the piecewise deterministic (mode-dependent) behavior gives the flow for the averaged component. The main theoretical result of Section \ref{SectionPDMP} (Theorem \ref{ThmMainPDMP}) gives the necessary and sufficient condition in order to guarantee near-viability of PDMP of switch-type. We discuss the implication on the simplified model of bacteriophage (cf. \cite{hasty_pradines_dolnik_collins_00}, \cite{crudu_debussche_radulescu_09}). To cope with the fundamental bistability issues, the authors of \cite{crudu_debussche_radulescu_09} invoke the quasi-steady distribution heuristics (resulting in some quartic equation cf. \cite[Eq. (26)]{crudu_debussche_radulescu_09} susceptible to give the lysogenic "steady-state"). The authors equally mention failure of such computations for nonlinear systems (cf. \cite{Mastnyetc2007}). For the hybrid model of phage $\lambda$, our results on invariance (or near-viability) turn out to prove an even stronger dichotomy: no (single) set of reaction speeds can guarantee the toggle between a lysogenic behavior (even when the latter is not reduced to a single steady-state but to a whole invariant region) an a lytic one (leading as usual to the attractor $0$). Indeed, starting from a clear lysogenic-associated initial configuration $x\in \mathring{K}^{lysogeny}$, the system can never reach the lysis-triggering zone ($\partial K^{lysogeny}$). This reinforces the biological belief that such toggle follows from an environmental (external) cue and should be explicitly taken into account in the modeling approach.\\\\
Let us briefly explain how the paper is organized. Section \ref{Prelim} presents the main notations employed throughout the paper as well as the invariance and near-viability notions. We also gather here some useful tools for signed distances. Section \ref{SectionNecessary} addresses the necessary local-Lipschitz condition for deterministic systems in order to expect invariance of open sets. The main result (Proposition \ref{PropMainNec}) is stated and accompanied by examples commenting the optimality of the condition. The main results on the near-viability of open sets with respect to Brownian-driven diffusions are stated in Section \ref{Section_NearViabDiff}. The ideas of proofs are outlined as well as some remarks concerning possible generalizations. Finally, the piecewise deterministic switch framework is presented in Section \ref{SectionPDMP}. We begin with a presentation of Hasty's model for bacteriophage lambda. We continue with structure considerations on PDMP and the statement of the theoretical result on near-viability (Theorem \ref{ThmMainPDMP}). Finally, we come back to the mathematical model in \cite{crudu_debussche_radulescu_09} and ensuing comments. All the proofs of the previously stated results are gathered in Section \ref{SectionProofs}.
\section{Preliminaries}
\label{Prelim}
\subsection{Notation, Definitions}
Throughout the paper we will make use of the following notations. \\
We denote by $\mathbb{R}_+$ the set of non-negative real numbers. For $x\in \mathbb{R}$, we let $x^+$ to be $x$ if $x\leq 0$ and $0$ otherwise (i.e. $x^+=\max\set{x,0}$). The \textit{absolute value} of $x\in\mathbb{R}$ is denoted by $\abs{x}$.\\
The state space (for our systems) is assumed to be some Euclidean space $\mathbb{R}^N$, where $N\geq 1$ is some positive integer. We denote by $\scal{\cdot,\cdot}$ the usual  Euclidean scalar product and by $\norm{\cdot}$ the induced Euclidean norm. \\
Whenever $O\subset\mathbb{R}^N$, 
\begin{itemize}
\item we denote its \textit{characteristic} ($0/1$-valued) function by $\mathbbm{1}_O$ (this notation also applies to general sets which are not subsets of an Euclidean space);
\item we denote by $O^c$ the \textit{complementary} set i.e. $\mathbb{R}^N\setminus O$, by $\mathring{O}$ its \textit{interior} and by $\bar{O}$ its \textit{closure} (these notation also apply to general subsets with an additional requirement of topology when necessary);
\item we denote by $d_O$ the \textit{Euclidean distance} function $d_O:\mathbb{R}^N\longrightarrow\mathbb{R}_+$ given by $d_O(x):=\inf_{y\in O}\norm{y-x},$ for all $x\in\mathbb{R}^N$, where $\inf$ stands for the infimum;
\item whenever $O$ is closed, the \textit{projection set} of $x\in\mathbb{R}^N$ (onto $O$) is given by \[\Pi_O(x):=\set{y\in O:\ d_O(x)=\norm{y-x}}.\] Whenever this set reduces to a singleton, we denote by $\pi_O(x)$ its unique element;
\item whenever $O$ is closed with nonempty interior $\mathring{O}$ and boundary $\partial O$, the \textit{oriented distance} from $\partial O$ is the function $\delta_O:\mathbb{R}^N\longrightarrow\mathbb{R}$ given by \[\delta_O(x):=d_{\partial O}(x)\mathbbm{1}_O(x)-d_{\partial O}(x)\mathbbm{1}_{O^c}(x),\textnormal{ for all }x\in\mathbb{R}^N;\]
\item whenever $O$ is open, we denote by \\
- $C^1(O)$ the set of real-valued differentiable functions on $O$ and by $D\phi$ the gradient of $\phi\in C^1(O)$,\\
- $C^2(O)$ the set of real-valued twice differentiable functions on $O$ and by $D^2\phi$ the Hessian matrix of $\phi\in C^2(O)$,\\
- $C^{2,1}(O)$ the subset of $C^2(O)$ of functions having bounded Lipschitz-continuous second order derivative.
\end{itemize} 
A set $K\subset \mathbb{R}^N$ is said to be 
\begin{itemize}
\item \textit{a closed domain on class $C^{2,1}$} if it is closed, connected and, for all $x\in\partial K$, there exists a positive radius $r>0$ and a regular function $\phi:\set{y \in \mathbb{R}^N:\ \norm{x-y}<r}\longrightarrow\mathbb{R}$ belonging to $C^{2,1}\pr{\set{y \in \mathbb{R}^N:\ \norm{x-y}<r}}$ such that \[\set{y\in \partial K:\ \norm{y-x}<r}=\set{y \in \mathbb{R}^N:\ \phi(y)=0};\]
\item \textit{a compact domain of class $C^{2,1}$} if it is a closed domain of class $C^{2,1}$ and compact.
\end{itemize}
The space of matrix of type $M\times N$ ($M$ lines and $N$ columns for positive integers $M,N$) with real entries is denoted by $\mathbb{R}^{M\times N}$. For a square matrix $A=\pr{a_{i,j}}_{1\leq i,j\leq N}\in \mathbb{R}^{N\times N}$, its trace is given by $Tr\pp{A}:=\sum_{1\leq i\leq N}a_{i,i}$. Finally, the space of \textit{symmetric matrix} is denoted by $\mathcal{S}^N$.\\
Finally, in Section \ref{Section_NearViabDiff} we are going to work with Lipschitz-driven controlled stochastic systems. We let $U$ denote some compact metric space and refer to $U$ as being the \textit{control space}.\\
Following \cite{Krylov_00}, we recall some notations. Given a bounded, Lipschitz-continuous function $f:\mathbb{R}^n\times U\longrightarrow \mathbb{R}$, we set
\begin{equation}
\label{HölderNorms}
\norm{f}_0:=\sup_{x\in \mathbb{R}^N, u\in U}\abs{f(x,u)},\ \pp{f}_1:=\sup_{x\in \mathbb{R}^N, y\in \mathbb{R}^N, x\neq y, u\in U}\frac{\abs{f(x,u)-f(y,u)}}{\norm{x-y}},\ \norm{f}_1:=\norm{f}_0+\pp{f}_1.
\end{equation}
When the function $f$ is $\mathbb{R}^N$-valued, the notations remain the same but $\abs{\cdot}$ is replaced with $\norm{\cdot}$. If the function $f$ is matrix $\mathbb{R}^{N\times M}$-valued, the norm is induced by $\sqrt{Tr\pp{ff^*}}$.\\
Let us now recall the concepts of viability and near- (or $\varepsilon$-)viability. To this purpose, we will consider the controlled trajectories of our systems (deterministic, diffusive, piecewise deterministic switched, etc.) denoted by $X_t^{x,u}$, starting (at time $t=0$) from $x\in \mathbb{R}^N$ and controlled by some (conveniently admissible) $U$-valued control process $u$. The notion of admissibility will be addressed in the corresponding sections. These processes will live on a probability space generically denoted by $\pr{\Omega,\mathcal{F},\mathbb{P}}$.
\begin{definition}
\begin{itemize}
\item[i.] The set $O\subset \mathbb{R}^N$ is said to be \textit{invariant} if, for every initial datum $x\in O$ and every admissible control process $u$, the associated trajectory satisfies $X_t^{x,u}\in O,\ dt\times\mathbb{P}-a.s.\ on\ \mathbb{R}_+\times \Omega$.\\
\item[ii.] The set $O$ is said to be \textit{(strongly) viable} if, for every initial datum $x\in O$ there exists an admissible control process $u$ such that the associated trajectory satisfies $X_t^{x,u}\in O,\ dt\times\mathbb{P}-a.s.\ on\ \mathbb{R}_+\times \Omega$.\\
\item[iii.] The set $O$ is said to be \textit{near-viable} or $\varepsilon$-\textit{viable} if, for every $x\in O$ and every $\varepsilon>0$, there exists a control process $u^\varepsilon$ such that the (exponential/discounted in time) probability of occupation of the exterior of $O$ is less than $\varepsilon$ i.e. \[\pr{\mathcal{E}xp(\lambda)\times\mathbb{P}}\pr{\set{\pr{t,X_t^{x,u^\varepsilon}}:t\geq 0,\ X_t^{x,u^\varepsilon}\in O^c }}=\lambda\mathbb{E}\pp{\int_{\mathbb{R}_+}e^{-\lambda t}\mathbbm{1}_{O^c}\pr{X_t^{x,u^\varepsilon}}dt}\leq \varepsilon.\]
\end{itemize}
\end{definition}
\begin{remark}
\begin{itemize}
\item[i.] The property of near-viability is equivalent to the value function \[v_O(x):=\inf_u\mathbb{E}\pp{\int_{\mathbb{R}_+}e^{-\lambda t}\mathbbm{1}_{O^c}\pr{X_t^{x,u^\varepsilon}}dt}\] being $0$ for all $x\in O$. Again,the infimum is taken over properly measurable $U$-valued process.\\
\item[ii.] The actual value of $\lambda>0$ does not change the notion, since the resulting measures are equivalent.
\item[iii.] Viability and  near-viability are equivalent under standard convexity assumptions on coefficients as soon as $O$ is closed (implying that the criterion $\mathbbm{1}_{O^c}$ is lower semi-continuous). For further details, the reader is invited to consult \cite{ElKarouiNguyenJeanblanc87} for diffusions, respectively \cite{Dempster_89} for piecewise deterministic Markov processes.
\item[iv.] If one wishes to express the invariance property through an optimal control problem, one only needs to replace the infimum in the definition of $v_O$ with supremum. 
\end{itemize}
\end{remark}
\subsection{Some Tools}
Following \cite{CannarsaDaPratoFrankowska2009}, we recall some elements linked to regular-border closed sets $K$. We assume $K\subset\mathbb{R}^N$ closed set with nonempty interior $\mathring K$ and boundary $\partial K$. The following proposition gathers some useful properties of compact domains of class $C^{2,1}$ (cf. \cite{CannarsaDaPratoFrankowska2009}, \cite{DelfourZolesio1994}).
\begin{proposition}
\label{PropToolsK}
Let us assume $K\subset\mathbb{R}^N$ to be a compact domain of class $C^{2,1}$. Then there exists some $\varepsilon_0>0$ such that 
\begin{itemize}
\item[i. ] For every $x\in K_{\varepsilon_0}:=\set{x\in K: \delta_K(x)\leq \varepsilon_0}$, the projection set $\Pi_{\partial K}(x)$ reduces to a unique element denoted by $\pi_{\partial K}(x)$ (i.e. $\norm{x-\pi_{\partial K}(x)}=\delta_K(x)$);
\item[ii. ] the function $\delta_K$ is of class $C^{2,1}(O)$ for some open set $O$ containing $K_{\varepsilon_0}$ and \[D\delta_K(x)=D\delta_K\pr{\pi_{\partial K}(x)}=-\nu_K\pr{\pi_{\partial K}(x)},\]where, $\nu_K(\cdot)$ stands for the outward unit normal to $K$.
\end{itemize}
Moreover, $\varepsilon_0>0$ can be chosen such that
\begin{itemize}
\item[iii. ] there exists a function $g\in C^{2,1}\pr{\mathbb{R}^N}$ satisfying $0\leq g(x)\leq 1, \textnormal{ for all }x\in K$, $0<g(x), \textnormal{ for all }x\in K\setminus K_{\varepsilon_0}$ and
$g(x)=\delta_K(x), \textnormal{ for all }x\in K_{\varepsilon_0}.$
\end{itemize}
\end{proposition}
\begin{remark}
Although stated and proven under the assumption of $K$ being a compact domain of class $C^{2,1}$, all our results of Sections \ref{Section_NearViabDiff} and \ref{SectionPDMP} are entirely based on Proposition \ref{PropToolsK} (and not on $K$ itself). They can be extended, without any particular effort to piecewise $C^{2,1}$-smooth domains (following \cite{CannarsaDaPratoFrankowska2009}).
\end{remark}
Unless stated otherwise, throughout the remaining of the paper, $K\subset\mathbb{R}^N$ will always stand for a compact domain of class $C^{2,1}$.
\section{On a Necessary Condition for the Invariance of Open Sets}
\label{SectionNecessary}
A key element for guaranteeing the invariance of $\mathring{K}$ (for regular-boundary $K$) under stochastic (uncontrolled) dynamics in \cite{CannarsaDaPratoFrankowska2009} is to give some a priori estimates on the infinitesimal operator applied to the signed distance function. To this purpose, the authors of \cite{CannarsaDaPratoFrankowska2009} deal with Lipschitz-continuous coefficients $b$ and $\sigma$. We will begin with discussing some examples of non-Lipschitz coefficients keeping the system in an open set. (To keep the arguments simple, we will assume that $\sigma$ vanishes.) This will show that the only element of interest for invariance is the (positive part of the) projection onto the exiting normal (and this locally close to the boundary).
\subsection{Specific Examples}
Let us now focus on the uncontrolled differential system
\begin{equation}
\label{EqB}
dX_t^x=b\pr{X_t^x}dt,\textnormal{ for all }t\geq 0;\ X_0^x=x\in \mathbb{R}^N.
\end{equation}
If the Lipschitz condition on the coefficient function is dropped, then invariance of closed sets might not imply invariance of their interior.
\begin{example}
Let us consider the one-dimensional, deterministic, uncontrolled coefficient $b:[
0,1]\longrightarrow\mathbb{R}$ given by $b(x):=\sqrt{\abs{1-x}}$, for all $x\in[0,1]$. It is obvious that the set $[0,1]$ is invariant. However, $(0,1)$ is not invariant (no even locally in time i.e. up to some space-independent $T>0$). Indeed, it is easy to see that $X_t^x=1$, for all $t\geq 2\sqrt{1-x}$. 
\end{example}
However, invariance of an open set may hold true even without $b$ being Lipschitz continuous. In the one-dimensional framework, this is the case, for instance, whenever the trajectories are systematically pointing towards the interior of the set.
\begin{example}
Let us consider the one-dimensional, deterministic, uncontrolled coefficient $b:\mathbb{R}\longrightarrow\mathbb{R}$ given by \[
b(x):=
-\sqrt{\abs{1-x}}\mathbbm{1}_{\pr{\frac{3}{4},\infty}}(x)+(-2x+1)\mathbbm{1}_{\pp{\frac{1}{4},\frac{3}{4}}}(x)+\sqrt{\abs{x}}\mathbbm{1}_{\pr{-\infty,\frac{1}{4}}}(x).
\]
Even though the coefficient is not Lipschitz-continuous, it is clear that for any interval $K:=[\alpha,\beta]$ such that $0\leq \alpha<\frac{1}{2}< \beta\leq 1$, the interior $(\alpha,\beta)$ is invariant. A simple computation yields that the modified coefficient 
$b^+(x):=\pr{b(x)\nu_K\pr{\pi_{\partial K}(x)}}^+$ is zero (thus being locally Lipschitz) on a neighborhood $\left[ \alpha, \max \pr{\alpha+\varepsilon,\frac{1}{4}}\right) \cup \left( \min\pr{\frac{1}{4},\beta-\varepsilon} ,\beta \right]$. Again, we have employed the notations for projection $\pi_{\partial K}$ and outward unit normal $\nu_K$.
\end{example}
\begin{remark}
\label{RemarkExp}
\begin{itemize}
\item[i.] Whenever $b$ is Lipschitz continuous on $K$, this type of continuity is quantified by \[\pp{b}_{1,K}:=\sup_{K\ni x\neq y \in K}\frac{\norm{b(y)-b(x)}}{\norm{y-x}}<\infty.\]
Thus, the local Lipschitz condition on $b$ (locally around $\partial K$) implies, in particular, \[\liminf_{\mathring{K}\ni x; \delta_K(x)\rightarrow 0}\frac{\norm{b^+(x)-b^+\pr{\pi_{\partial K}(x)}}}{\norm{x-\pi_{\partial K}(x)}}=\liminf_{\mathring{K}\ni x; \delta_K(x)\rightarrow 0}\frac{\norm{b^+(x)-b^+\pr{\pi_{\partial K}(x)}}}{\delta_K(x)}<\infty.\] In this case, one can actually replace $\liminf$ with $\limsup$.
\item[ii.] It may be interesting to note that, although the previous system is driven by a non-Lipschitz coefficient, one can construct a system $X^{x,K}$ driven by a Lipschitz coefficient $b_K$ such that $\mathring{K}$ is invariant w.r.t. $X^{x,K}$ and $\delta_{\partial K}\pr{X_t^{x,K}}\leq \delta_{\partial K}\pr{X_t^{x}}$ for all $0\leq t$ (thus providing some kind of dominating behavior). A simple modification leading to this behavior is obtained by taking \[b_{[\alpha,\beta]}(x):=0\times \mathbbm{1}_{\set{\frac{3}{4}<x}\cup\set{\frac{1}{4}>x}}+\pr{4x-3}\times \mathbbm{1}_{x\in\pp{\frac{2}{3},\frac{3}{4}}}+\pr{-2x+1}\times \mathbbm{1}_{x\in\pr{\frac{1}{3},\frac{2}{3}}}+\pr{4x-1}\times \mathbbm{1}_{x\in\pp{\frac{1}{4},\frac{1}{3}}}.\] If this kind of dominated modification can be constructed, the condition in i. only needs to be checked on the modification (and not on the initial system).
\end{itemize}
\end{remark}
\subsection{A Local Lipschitz-like Necessary Condition}
From now on, we assume $b$ to be bounded and uniformly continuous. This guarantees the existence of solution(s) to (\ref{EqB}). Whenever the solution fails to be unique, invariance is to be understood as every (local) solution to (\ref{EqB}) starting from some (arbitrary) $x\in O$ to stay in $O$ (or for as long as it is defined in other frameworks). Before proving the necessity of a local Lipschitz-like condition (similar to item i in Remark \ref{RemarkExp}), we set, for every $x\in K$ such that $\delta_K(x)\leq\varepsilon_0$,
\begin{equation}
\label{BPlus}
b^+(x)={\scal{b(x),\nu_K\pr{\pi_{\partial K}(x)}}}^+=\max\set{\scal{b(x),\nu_K\pr{\pi_{\partial K}(x)}},0}.
\end{equation}
For every $\varepsilon\in (0,\varepsilon_0]$, we define \begin{equation}
\label{DefZeta}\zeta\pr{\varepsilon}=\inf_{x\in \mathring{K}, \delta_K(x)\leq\varepsilon}\frac{\abs{b^+(x)-b^+\pr{\pi_{\partial K}(x)}}}{\delta_K(x)}.\end{equation}
\begin{proposition}
\label{PropZeta}
The function $\zeta:(0,\varepsilon_0]\longrightarrow\mathbb{R}_+$ is a non-increasing, left-continuous function.
\end{proposition}
The proof is quite standard. For our readers' sake, we provide it in Section \ref{SectionProofs}.\\
\begin{remark}\label{RemZetaConnect}As careful look at the proof of the necessary condition (Proposition \ref{PropMainNec}), these functions $\zeta$ can either be defined globally or with respect to each connected component of $\partial K$.\end{remark}
Whenever $\lim_{\varepsilon\rightarrow 0}\inf_{x\in \mathring{K}, \delta_K(x)\leq\varepsilon}\frac{\abs{b^+(x)-b^+\pr{\pi_{\partial K}(x)}}}{\delta_K(x)}=\infty,$ the function $F_\zeta:\left[\frac{1}{\varepsilon_0},\infty\right)\longrightarrow [0,1]$ given by \[F_\zeta(r)=1-\frac{\zeta\pr{\varepsilon_0}}{\zeta\pr{\frac{1}{r}}}\textnormal{, for all }r\in \left[\frac{1}{\varepsilon_0},\infty\right),
\] is the cumulative distribution function of some random variable (one may, eventually, want to extend $F_\zeta$ by asking it to be null for arguments not exceeding $\frac{1}{\varepsilon_0}$). We let $F_\zeta^{-1}$ denote the usual generalized inverse i.e. \[F_\zeta^{-1}:[0,1]\longrightarrow\left[\frac{1}{\varepsilon_0},\infty\right),\ F_\zeta^{-1}(p):=\inf\set{r:r\in\left[\frac{1}{\varepsilon_0},\infty\right), F_\zeta(r)\geq p}\textnormal{, for all }p\in [0,1].\]
\begin{proposition}\label{PropMainNec}
If $b$ is continuous and $\mathring{K}$ is invariant with respect to $\ref{EqB}$ then either 
\begin{itemize}
\item[1.] the function $\zeta$ is bounded i.e. \[\liminf_{x\in\mathring{K};\ \delta_K(x)\rightarrow 0}\frac{\abs{b^+(x)-b^+\pr{\pi_{\partial K}(x)}}}{\abs{x-\pi_{\partial K}(x)}}<\infty \ \pr{ \textnormal{or, equivalently, } \liminf_{x\in\mathring{K};\ \delta_K(x)\rightarrow 0}\frac{b^+(x)}{\delta_K(x)}<\infty},\] \textnormal{or}
\item[2.] the inverse $F_\zeta^{-1}$ grows very fastly to $\infty$ as its argument grows to $1$ i.e. \begin{equation}
\label{AssEta}
\inf_{\beta>1}\limsup_{\delta\rightarrow 0}\delta\pp{-\ln{\delta}}^\beta \ln{F_\zeta^{-1}\pr{1-\delta}}=\infty.
\end{equation}
\end{itemize} 
\end{proposition}
We postpone the proof of this assertion to Section \ref{SectionProofs}. The simple underlying idea is that if $b^+$ fails to be locally Lipschitz (i.e. $\zeta$ is unbounded) and provided the continuity modulus  is smooth enough at $0+$ ((\ref{AssEta}) fails to hold), then, for any finite time horizon $T>0$, by starting close enough to $\partial K$, the trajectory reaches $\partial K$ prior to $T$.\\
Let us comment on the qualitative interpretation of our result.
\begin{remark}
\begin{itemize}
\item[i.] If the coefficient function is Lipschitz continuous, $\zeta$ is bounded by the Lipschitz constant and the result is trivial.
\item[ii.] The condition 2. (or negation of) should be seen as a regularity assumption on the continuity modulus of $b^+$. For large continuity moduli the condition 2.  is never satisfied. For such systems, invariance of open sets implies the local Lipschitz-like condition in 1. \\
(a) Let us consider, as example, the case when $b^+$ is (supra) Hölder-continuous function (say, with Hölder coefficient $\kappa<1$), i.e., locally (close to $\partial K$) $\abs{b^+(x)-b^+(y)}\geq c\norm{x-y}^\kappa$. One easily computes (assuming $c=1$),
\begin{equation*}
\begin{split}
&\zeta(\varepsilon)\geq\varepsilon^{\kappa-1}, F_\zeta(r)\geq1-\zeta\pr{\varepsilon_0}r^{\kappa-1}, F_\zeta^{-1}(p)\leq\pr{\frac{1-p}{\zeta\pr{\varepsilon_0}}}^{\frac{1}{\kappa-1}}\\
&\limsup_{\delta\rightarrow 0}\delta\pp{-\ln{\delta}}^\beta\ln\pr{ F_\zeta^{-1}\pr{1-\delta}}=\limsup_{\delta\rightarrow 0}\delta\pp{-\ln{\delta}}^{\beta+1}=0,\textnormal{ for all }\beta>1.
\end{split}
\end{equation*}
(b) Another example is provided by functions of type $b(x)=x\pr{\ln\frac{1}{x}}^\alpha \times \mathbbm{1}_{x>0}+0\times \mathbbm{1}_{x=0}$, for some $\alpha>1$. One computes  
\begin{equation}
\label{LimitExp}
\begin{split}
&\zeta(\varepsilon)=\pr{\ln\frac{1}{\varepsilon}}^\alpha, F_\zeta(r)=1-\frac{\zeta\pr{\varepsilon_0}}{\pr{\ln r}^\alpha}, F_\zeta^{-1}(p)=e^{\pr{\frac{\zeta\pr{\varepsilon_0}}{1-p}}^\frac{1}{\alpha}}\\
&\limsup_{\delta\rightarrow 0}\delta\pp{-\ln{\delta}}^\beta\ln\pr{ F_\zeta^{-1}\pr{1-\delta}}=\limsup_{\delta\rightarrow 0}\delta^{1-\frac{1}{\alpha}}\pp{-\ln{\delta}}^{\beta}<\infty,\textnormal{ for all }\beta>1.
\end{split}
\end{equation}
Finally, the condition (\ref{AssEta}) is valid as soon as $\alpha\leq 1$. Our local Lipschitz-condition can no longer be guaranteed (at least with our proof !) for better continuity moduli  (see also Example \ref{Exp_xlnx}).
\item[iii.] Sharper dichotomy can be achieved by replacing, in the definition of $\zeta$, the denominator $\delta_K(x)$ with $\Phi\pr{\delta_K(x)}$.  In this case, the estimates for $\frac{1}{\varepsilon^j}$ in the proof of Proposition \ref{PropMainNec} will no longer be based on Gronwall-type inequality but on generalizations. The reader should get a good idea of conditions on $\Phi$ by recalling Osgood-type uniqueness or taking a look at \cite{LaSalle1949}.
\item[iv.] In the Brownian diffusion case, the same proof holds true by replacing $\scal{b(x),\nu_K\pr{\pi_{\partial K}(x)}}$ with $-\mathcal{L}\delta_K(x)$, where $\mathcal{L}$ is the infinitesimal generator associated to the diffusion. Of course, further assumptions on the continuity moduli should be present in order to guarantee existence of a continuous solution(e.g. \cite{FangZhang2003}). 
\end{itemize}
\end{remark}
\subsection{Some Further (Counter)Examples}
In this subsection, we will give some comments (via explicit examples) on (quasi) optimality of our conditions. First, we emphasize the fact that the local Lipschitz-like condition $\liminf_{\varepsilon\rightarrow 0} \zeta(\varepsilon)$ cannot (in all generality) be replaced with $\limsup_{\varepsilon\rightarrow 0}$. Second, we point out the (quasi) optimality of the regularity assumption (i.e. in order for invariance to imply local Lipschitz-like condition, (\ref{AssEta}) should fail to hold). We present an example for which the continuity modulus is very slow (translating into $F_\zeta^{-1}(1-p)$ of type $\frac{1}{p\ln\frac{1}{p}}$, thus implying (\ref{AssEta})) and invariance for an open set holds true in the absence of a local Lipschitz-like condition ($\liminf_{\varepsilon\rightarrow 0} \zeta(\varepsilon)<\infty$). Finally, we present an example showing that the construction of $\zeta$ has to be at least on some neighborhood of points $x\in\partial K$ and it cannot be replaced with $\liminf_{x\in \mathring{K}; x\rightarrow \bar{x}}\frac{b^+(x)}{\delta_K(x)}$.
\begin{example}
\label{ExpLiminfLimSup}
Let us consider the following (one-dimensional) coefficient function
\[
b(x):=\left\lbrace
\begin{split}
-\left\lvert \sqrt{x} \sin\pr{\frac{1}{x}} \right\rvert,&\textnormal{ if }0<x \leq 1,\\
0,&\textnormal{ if }x=0.
\end{split}
\right.
\]
It is obvious that $\pr{0,1}$ is invariant. Indeed, if $x\in\pr{0,1}$, then there exists $k\in \mathbb{N}$ such that $\frac{1}{(k+1)\pi}\leq x$. Then $X_t^x\in \pp{\frac{1}{(k+1)\pi},x}$ which implies invariance. However, $\limsup_{x\rightarrow 0+}\frac{b^+(x)}{\abs{x}}=\limsup_{x\rightarrow 0+}\frac{\abs{\sin\frac{1}{x}}}{\sqrt{x}}=\infty$. Therefore, the local Lipschitz-like condition $\liminf_{x\rightarrow 0+}\frac{b^+(x)}{\abs{x}}<\infty$ holds but $\liminf$ cannot be replaced with $\limsup$.
\end{example}
Let us now show that this condition on the continuity modulus (with $F_\zeta^{-1}(1-p)$ of type $\frac{1}{p\ln\frac{1}{p}}$) is nearly optimal in order to get local Lipschitz-like behavior.
\begin{example}
\label{Exp_xlnx}
We consider the set $K=[0,1]$ and the coefficient function $b(y)=-\frac{y}{b^0(y)}$, if $y>0$ (and $b(0)=0$), where  $b^0(y)$ is the unique solution of the equation\[\pr{b^0(y})^{b^0(y)}=e^{\frac{1}{\ln y}}.\] It is clear that $b^0$ is increasing, $lim_{y\rightarrow 0}b^0(y)=0$ and $b$ is continuous. 
One notes that $b$ is decreasing. Indeed, by changing the variable $t:=b^0(y)$, one gets a function $\tilde{b}(t):=-\frac{e^{\frac{1}{t\ln t}}}{t}$ whose derivative is negative (for $t$ small enough). 
If $x=n^{-1}$ (small enough) is the initial datum of our equation driven by $b$, then $X^{x}_t> \frac{1}{n+1}$ for all $t< \frac{1}{n+1}b^0\pr{\frac{1}{n}}$. Hence, the time $t_0^n$ to reach $0$ starting from $x=n^{-1}$ satisfies the lower bound
\[t_0^n\geq \sum_{k\geq n} \frac{1}{k+1}b^0\pr{\frac{1}{k}}>\sum_{k\geq n} \frac{1}{(k+1)\ln k},\] where we have used $b^0(y)> \frac{-1}{\ln y}$ which is obtained using the monotonicity of $b^0$ and the fact that 
\[ \frac{-1}{\ln y}\ln \pr{  \frac{-1}{\ln y}}< \frac{1}{\ln y}.\]Therefore, since $t_0^n=\infty$, for every $n\geq 1$, the set $(0,1)$ is invariant ($X_t^x$ decreasing in time).\\
On the other hand, standard computations (around $0$ for the quantities $\zeta,\ F_\zeta,\ F_\zeta^{-1}$) yield \[\begin{split} &\zeta(\varepsilon)=\inf_{y\leq \varepsilon}\frac{1}{b^0(y)}=\frac{1}{b^0(\varepsilon)},\ F_\zeta(r)=1-\frac{b^0\pr{r^{-1}}}{b^0\pr{\varepsilon_0}},\\ &\ln F_\zeta^{-1}(1-p)=\frac{-1}{b^0\pr{\varepsilon_0}p\ln\pr{b^0\pr{\varepsilon_0}p}}, \lim_{p\rightarrow 0}p\pr{\ln p}^\beta \ln F_\zeta^{-1}(1-p)=\infty,\ \forall \beta>1.\end{split}\] 
\end{example}
Finally, a careful look at the proof of Proposition \ref{PropMainNec} (see Section \ref{SectionProofs}) shows that $\zeta$ does not need to be constructed globally but only with respect to neighborhoods of connected components of $\partial K$. This is obvious enough especially when one tries to apply the proof to a one-dimensional example $K=\pp{\alpha_1,\alpha_2}$. Then one can reason either on a neighborhood of $\alpha_1$ or with respect to $\alpha_2$.
Since $K$ is compact, whenever $\liminf_{\mathring{K}\ni x,\ \delta_K\pr{x}\rightarrow 0}\frac{b^+(x)}{\delta_K\pr{x}}<\infty$, then there exists at least one $\bar{x}\in \partial{K}$ such that $\liminf_{\mathring{K}\ni x\rightarrow \bar{x}}\frac{b^+(x)}{\delta_K\pr{x}}<\infty$. One might be tempted to ask for the stronger condition : 
for every $\bar{x}\in \partial{K}$, one has \[\liminf_{\mathring{K}\ni x\rightarrow \bar{x}}\frac{b^+(x)}{\abs{x-\bar{x}}}<\infty.\] However, one can construct an example in which the (open) unit ball in $\mathbb{R}^2$ is invariant and, yet, the previous condition does not hold for some $\bar{x}$ such that $\norm{\bar{x}}=1$. 
\begin{example}
\label{ExpPolar}
We describe the dynamics in polar coordinates $(\rho,\theta)$ such that
\[\left\lbrace \begin{split}d\rho_t&=\sqrt{1-\rho_t}\pr{\theta_t}^+\pr{\frac{\pi}{2}-\theta_t}^+dt\\
d\theta_t&=\frac{1}{1-\rho_t}\pr{\theta_t}^+\pr{\frac{\pi}{2}-\theta_t}^+dt. \end{split}\right.\] As usual, $\theta^+=\max\set{\theta,0}$.
First, it is clear that $b^+\pr{\rho\cos(\theta),\rho\sin(\theta)}=\dot{\rho}=\sqrt{1-\rho}\pr{\theta_t}^+\pr{\frac{\pi}{2}-\theta_t}^+$ and, thus, \[\liminf_{\pr{x,y}\rightarrow\pr{\frac{\sqrt{2}}{2},\frac{\sqrt{2}}{2}}}\frac{b^+\pr{x,y}}{1-\sqrt{x^2+y^2}}=\liminf_{\rho\rightarrow 1}\frac{\pi^2}{16\sqrt{1-\rho}}=\infty.\]
Second, starting from some point $\pr{\rho_0,\theta_0}\in [0,1)\times \pr{0,\frac{\pi}{2}}$, let us assume that the trajectory reaches the unitary circle at some time $t_{min}$. Prior to $t_{min}$, both $\rho$ and $\theta$ are non decreasing. Let us assume that $\theta_t<\frac{\pi}{2}$ for $t<t_{min}$. Then, we get the equality $\dot{\theta}=\frac{\dot{\rho}}{\pr{1-\rho}^{\frac{3}{2}}}$ for $0\leq t<t_{min}$ which gives\[\frac{\pi}{2}>\theta_t=\theta_0-2\pr{1-\rho_0}^{-\frac{1}{2}}+2\pr{1-\rho_t}^{-\frac{1}{2}}.\] This implies that $\rho_t<1-\pr{\frac{\pi}{4}+\pr{1-\rho_0}^{-\frac{1}{2}}}^{-2}$ and, thus, the viability of $\left\lbrace(x,y):x^2+y^2<1\right\rbrace$ follows. \\
Finally, we wish to point out that, although $b$ is not (globally) bounded (because of the presence of the term $\frac{1}{1-\rho}$ as $\rho$ is close to $1$), the previous upper estimate on $\rho_t$ shows that for each initial data (in $\left\lbrace(x,y):x^2+y^2<1\right\rbrace$), one actually deals with a bounded coefficient.
\end{example}

\section{Border Avoidance for Controlled Diffusions}\label{Section_NearViabDiff}
We consider a probability space $\pr{\Omega,\pr{\mathcal{F}_t}_{t\geq 0},\mathbb{P}}$ endowed with a filtration satisfying the usual assumptions on right-continuity and completeness. On this space, we consider the $d$-dimensional Brownian motion $W$. We will deal with the near-viability of Brownian-driven systems\begin{equation}
\label{CtrlSDE0}
\left\lbrace 
\begin{split} 
dX_t^{x,u}&=b\pr{X_t^{x,u},u(t)}dt+\sigma\pr{X_t^{x,u},u(t)}dW_t,\textnormal{ for }t\geq0,\\
X_0^{x,u}&=x\in \mathbb{R}^N.
\end{split}
\right.
\end{equation}
From now on, unless stated otherwise, we assume that the coefficients $b$ and $\sigma$ are uniformly continuous on $\mathbb{R}^N\times U$ and (with the notation (\ref{HölderNorms}), see also the following lines on vectorial notations) $\norm{b}_1+\norm{\sigma}_1<\infty$ (i.e. $b$ and $\sigma$ are bounded and Lipschitz-continuous in space). The space of admissible control processes will be denoted by $\mathcal{U}$ and it consists of progressively measurable $U$-valued processes $u$. Under these assumptions, existence and uniqueness of a (strong) controlled solution to (\ref{CtrlSDE0}) is standard.\\
For $\delta>0$ and progressively measurable controls $e$ taking their values in unit ball of $\mathbb{R}^N$ (space denoted by $\mathcal{E}$), respectively $u\in\mathcal{U}$, we consider the augmented control system (appearing in Krylov's method of shaking the coefficients cf. \cite{Krylov_00}).
\begin{equation}
dX_t^{x,u,e,\delta}=b\pr{X_t^{x,u,e,\delta}+\delta^2e_t,u_t}dt+\sigma\pr{X_t^{x,u,e,\delta}+\delta^2e_t,u_t} dW_t,\textnormal { for } t\geq 0,\ X_0^{x,u,e,\delta}=x\in \mathbb{R}^N.
\end{equation}
We begin with recalling some standard estimates result (see \cite[Page 11]{Krylov_00} for the second assertion).
\begin{proposition}
\label{PropKrylov}
There exists a constant $\lambda_0$ depending only on $\norm{b}_1$ and $\norm{\sigma}_1$ such that
\begin{equation}
\begin{split}
&\textnormal{i. } \mathbb{E}\pp{\sup_{t\in [0,T]}\norm{X_t^{x,u}}^2dt}\leq \lambda_0e^{\lambda_0T}\pr{1+\norm{x}^2},\\
&\textnormal{ii. } \mathbb{E}\pp{\sup_{t\in [0,T]}\norm{X_t^{x,u,e,\delta}-X_t^{x,u}}}\leq\lambda_0e^{\lambda_0T}\delta  \textnormal { and}\\
&\textnormal{iii. } \mathbb{E}\pp{\sup_{t\in [0,T]}\norm{X_t^{x,u,e,\delta}-X_t^{y,u,e,\delta}}}\leq\lambda_0e^{\lambda_0T}\norm{x-y},
\end{split}
\end{equation}for all $\delta>0, x\in \mathbb{R}^N,y\in \mathbb{R}^N$ and all control processes $(u,e)\in \mathcal{U}\times\mathcal{E}$  i.e progressively measurable processes $u$ taking their values in $U$ and $e$ taking their values in the unit ball of $\mathbb{R}^N$.
\end{proposition}
\begin{remark}
\label{RemLipschitzOsgood}
These estimates essentially follow from Gronwall's inequality (hence the Lipschitz-type assumptions). Alternatively, one can hope to improve the assumptions by asking a Osgood-type condition of (non-)ntegrability for the continuity moduli or generalizations (e.g. \cite{LaSalle1949}), thus allowing use of non-Lipschitz coefficients. In this framework, a careful look at \cite[Proof of Theorem D]{FangZhang2005} shows that near-viability up to a fixed finite horizon $T>0$ is more adequate.
\end{remark}
Second, we introduce the notations 
\begin{equation}
\label{csigmacL}
\left\lbrace
\begin{split}
&\sup_{u\in U, x\in \mathring{K} \ s.t.\  \delta_K(x)\leq \varepsilon_0}\frac{\norm{\sigma(x,u)-\sigma(\pi_{\partial K}(x),u)}}{\delta_K(x)}=:c_\sigma,\\
&\sup_{u\in U, x\in \mathring{K} \ s.t.\  \delta_K(x)\leq \varepsilon_0}\frac{\abs{\mathcal{L}^u\delta_K(x)-\mathcal{L}^u\delta_K\pr{\pi_{\partial K}(x)}}}{\delta_K(x)}=:c_\mathcal{L},
\end{split}
\right.
\end{equation}
where, as usual, $\mathcal{L}^u\varphi(x):=\frac{1}{2}Tr\pp{\sigma(x,u)\sigma^*(x,u)D^2\varphi(x)+\scal{b(x,u),D\varphi(x)}}$ for all $x\in\mathbb{R}^N,\ u\in U$ and for regular twice differentiable functions $\varphi\in C^2\pr{\mathbb{R}^N}$.
\begin{remark}
\label{RemarkcscL}
In our Lipschitz framework, both constants $c_\sigma$ and $c_\mathcal{L}$ are upper-bounded by a (generic) constant that only depends on $\norm{b}_1$ and $\norm{\sigma}_1$. However, most of our proofs can be generalized to non-Lipschitz settings (provided further assumptions, see Remark \ref{RemGenNonLipsch}). In this case, the conditions (\ref{csigmacL}) are to be regarded as a uniform-in-control upper limit condition much like the necessary one for invariance (cf. Proposition \ref{PropMainNec} 1.). One should ask for \begin{equation}
\label{csigmacL2}
\left\lbrace\begin{split}\limsup_{\delta_K(x)\rightarrow 0}\sup_{u\in U}\frac{\norm{\sigma(x,u)-\sigma(\pi_{\partial K}(x),u)}}{\delta_K(x)}<\infty,\\
\limsup_{\delta_K(x)\rightarrow 0}\sup_{u\in U}\frac{\abs{\mathcal{L}^u\delta_K(x)-\mathcal{L}^u\delta_K\pr{\pi_{\partial K}(x)}}}{\delta_K(x)}<\infty,
\end{split}
\right.
\end{equation}
Furthermore, the second inequality is only needed to guarantee \[\mathcal{L}^u\delta_K(x)-\mathcal{L}^u\delta_K\pr{\pi_{\partial K}(x)}\geq -c_\mathcal{L}\delta_K(x)\] and only when $\mathcal{L}^u\delta_K(x)\leq 0$. Hence, one can actually use only $\pr{-\mathcal{L}^u\delta_K(x)}^+$ as in the necessary condition.
\end{remark}
From now on, we will denote by $\lambda$ a large enough constant such that\begin{equation}
\label{Ass_lambda}
\lambda>2\lambda_0+c_\sigma^2+ c_\mathcal{L}.
\end{equation}
\begin{remark} For further developments, we will need further details on how large the constant $\lambda$ should be chosen. Since $K$ is compact, its diameter $diam(K):=\sup\set{\norm{x-y}:(x,y)\in K\times K}$ is finite. Classical estimates on solutions of SDEs yield (due to assumptions on coefficients), the existence of a constant $k$ such that, for all initial data $x\in \mathbb{R}^N$, any admissible control $u$ and any $t>0$, \[\mathbb{E}\pp{\sup_{0\leq s\leq t}\norm{X_s^{x,u}-x}^2}\leq kt.\]It follows that \[\mathbb{P}\pr{\sup_{0\leq s\leq t}\norm{X_s^{x,u}-x}\geq\frac{\varepsilon_0}{2}}\leq \frac{4k}{\varepsilon_0^2}t.\] One then picks $t^*=\frac{\varepsilon_0^3}{32k\times diam(K)}$ and $\lambda>0$ such that 
\begin{equation}
\label{Ineqlambda1}
\left\lbrace
\begin{split}
\mathbb{P}\pr{\sup_{0\leq s\leq t^*}\norm{X_s^{x,u}-x}\geq \frac{\varepsilon_0}{2}}&\leq \frac{\varepsilon_0}{8diam(K)};\\
e^{- \lambda t^*} \leq \frac{\varepsilon_0}{8diam(K)}.
\end{split}
\right.
\end{equation}
\end{remark} 
We recall the viability indicator value function $V:\mathbb{R}^N\longrightarrow\mathbb{R}_+$, given by \begin{equation}
V(x):=\inf_{u\in\mathcal{U}}\mathbb{E}^x\pp{\int_0^\infty e^{-\lambda t}\mathbbm{1}_{\pr{\mathring{K}}^c}\pr{X_t^{x,u}}dt},\textnormal{ for all }x\in \mathbb{R}^N.
\end{equation}
To prove the equivalence between the near-viability of the closed set $K$ and the near-viability property for $\mathring{K}$, one proceeds in three steps. First, we consider the Lipschitz-continuous sup-convoluted approximations of $\mathbbm{1}_{\pr{\mathring{K}}^c}$ (denoted by $f_n$ for $n\geq 1$) and the (Lipschitz-continuous, infinite-horizon, $\lambda$-discounted) associated value functions $V_n$. The first result (Proposition \ref{VnConvergeToV}) gives the pointwise convergence of the approximating functions. \\
Second, we prove (in Lemma \ref{LemmaReductionToKeps}) that, in order for $V$ to be null on $\mathring{K}$, one only needs to focus on (arbitrarily small) neighborhoods of the boundary $\partial K$. In particular, the study can be reduced to a set where the signed distance $\delta_K$ is smooth.\\
Third, we use Krylov's shaking of coefficients method to construct regular subsolutions ($V_n^{\frac{1}{n^2}}$) of the Hamilton-Jacobi-Bellman equation satisfied (in a viscosity sense) by $V_n$. By multiplying these functions with $\delta_K$, one gets a family of smooth functions $W_n:=-\delta_K\times V_n^{\frac{1}{n^2}}$. By "comparing" (at global minimum points) these functions with the 0-constant (subsolution) and passing to the limit as $n\rightarrow\infty$,  it follows that $V$ cannot be strictly positive at any point $x\in\mathring{K}$. This comparison uses extensively the (closed-set) viability characterization in \cite{BCQ2001} and the inequality \begin{equation}
\label{C0}
\mathcal{L}^u\delta_K(x)-\mathcal{L}^u\delta_K\pr{\pi_{\partial K}(x)}\geq -c_\mathcal{L}\delta_K(x).
\end{equation}
Let us now consider the (approximating) functions \[f_n(x):=\pr{1-nd_{\pr{\mathring{K}}^c}(x)}^+.\]
These functions form a non-increasing sequence of Lipschitz-continuous applications that converges to $f(x)=\mathbbm{1}_{\pr{\mathring{K}}^c}(x)$.
Therefore, for every $\lambda>0$, the value functions
\begin{equation}
\label{Vn}
V_n(x):=\inf_{u\in\mathcal{U}} \mathbb{E}^x\pp{\int_0^\infty e^{-\lambda t}f_n\pr{X_t^{x,u}}dt}, \textnormal{ for all }x\in \mathbb{R}^N
\end{equation}
satisfy $V_n\geq V$. In fact, we prove the following convergence.
\begin{proposition}
\label{VnConvergeToV}
The approximating functions $V_n$ converge (pointwise) to the (viability indicator) function $V$.
\end{proposition}
The proof is rather standard and postponed to Section \ref{SectionProofs}.
The main result of the section gives the equivalence between the near-viability of $K$ and the near-viability of its interior $\mathring{K}$.
\begin{thm}
\label{ThmMainDiff}
For every controlled diffusion driven by bounded, Lipschitz-continuous coefficients $b,\sigma$, the set $K$ is near-viable if and only if $\mathring{K}$ enjoys this property. 
\end{thm}
The proof (postponed to Section \ref{SectionProofs}) relies on the following program. Using Krylov's shaking of coefficients method (cf. \cite{Krylov_00}), we exhibit a family of regular functions $V_n^{\frac{1}{n^2}}$ such that 
\begin{itemize}
\item $V_n^{\frac{1}{n^2}}-V_n$ converges to $0$ (uniformly in space),
\item $V_n^{\frac{1}{n^2}}$ satisfy, in classical subsolution sense the HJB equation associated to $V_n$. 
\end{itemize}
Second, the function $W_n$ defined as $W_n(x):=-V_n^{\frac{1}{n^2}}\delta_K(x)$ is shown to satisfy a certain super-solution condition (at local minimum points). This step strongly relies on the regularity of $\delta_K$ (to which the results on closed-sets viability of \cite{BCQ2001} apply) and the constants $c_\sigma$ and $c_\mathcal{L}$ being finite (but not on the Lipschitz continuity itself). Finally, basically by comparing (at the limit as $n\rightarrow\infty$) this supersolution with $0$, it is shown that the global minimum (of $-V\delta_K$) cannot be other than $0$, thus proving that $V$ is zero (i.e. near-viability of the open set $\mathring{K}$).\\
\begin{remark}
\label{RemGenNonLipsch}
\begin{itemize}
Provided such functions $V_n^{\frac{1}{n^2}}$ can be constructed, the Lipschitz assumptions on $b,\  \sigma$ can be dropped and the conclusion (cf. Steps 2,3 in the proof of Theorem \ref{ThmMainDiff}) still holds true.
\end{itemize}
\end{remark}
Finally, to guarantee that the global minimum points we have talked about before are not very far from the boundary, we need the following result.
\begin{lem}
\label{LemmaReductionToKeps}
Let us assume that the set $K$ is near-viable. \\
i. If $x\in K\setminus K_\frac{\varepsilon_0}{2}$, then \[ V_n(x)\leq \max_{y\in K^{\frac{\varepsilon_0}{2}}}V_n(y)\] and the equality holds true only if the right-hand member is $0$.\\
ii. Whenever the restriction of $V$ to $K_\frac{\varepsilon_0}{2}\setminus\partial K$ is null, $V$ is constantly equal to $0$ on $\mathring{K}$.
\end{lem}
As for the other results, the proof is postponed to Section \ref{SectionProofs}. The proof of the first assertion uses the dynamic programming principle (for the approximations $V_n$). The second one follows from the first and needs some kind of uniform convergence to $0$.
\section{An Application to Switched Models of Gene Networks}
\label{SectionPDMP}
\subsection{Hasty's Model of Bacteriophage}
For complex systems of interaction between several players or species, the notion of multi-stability plays a central part. We have in mind the simplest example relevant to this framework in the hybrid modeling of gene networks: the phage lambda. Roughly speaking, in order to survive, the bacteriophage relies on the host (E-Coli). Two issues are then possible (for a temperate form of phage): either a cohabitation (lysogeny resulting in moderate replication of prophage) or high speed replication ultimately leading to excision (lysis). Following the simplified model introduced by \cite{hasty_pradines_dolnik_collins_00}, a piecewise deterministic model has been introduced in \cite{crudu_debussche_radulescu_09} (and the averaging probabilistic techniques rigorously justified in \cite{crudu_Debussche_Muller_Radulescu_2012}). To cope with the fundamental bistability issues, the authors invoke the quasi-steady distribution heuristics (resulting in some quartic equation cf. of \cite[Eq. (26)]{crudu_debussche_radulescu_09} susceptible to give the lysogenic "steady-state"). The authors equally mention failure of such computations for nonlinear systems (cf. \cite{Mastnyetc2007}).
For the hybrid model of phage $\lambda$, our results on invariance (or near-viability) turn out to prove an even stronger dichotomy: no (single) set of reaction speeds can guarantee the toggle between a lysogenic behavior (even when the latter is not reduced to a single steady-state but to a whole invariant region) an a lytic one (leading as usual to the attractor $0$). This reinforces the biological belief that such toggle follows from an environmental (external) cue.\\
The reactions system corresponding to the model reduction in \cite{hasty_pradines_dolnik_collins_00} is the following.
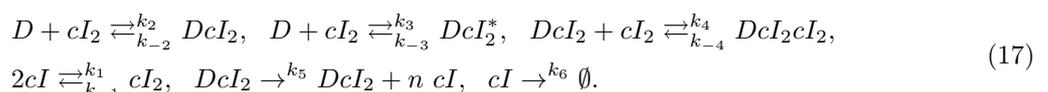
\begin{equation}
\label{HastySyst}
\begin{split}
&D+cI_2\rightleftarrows_{k_{-2}}^{k_2}DcI_2,\ \ D+cI_2\rightleftarrows_{k_{-3}}^{k_3}DcI_2^*,\ \ DcI_2+cI_2\rightleftarrows_{k_{-4}}^{k_4}DcI_2cI_2,\\ 
&2cI\rightleftarrows_{k_{-1}}^{k_1}cI_2,\ \ DcI_2\rightarrow^{k_5}DcI_2+n\ cI,\ \ cI\rightarrow^{k_6}\emptyset.
\end{split}
\end{equation}
Let us now give some brief explanations on these equations. The first line in (\ref{HastySyst}) corresponds to DNA (D) bindings along promoter, repressor or both sites. These will only indicate a functioning mode $\Gamma$ (whose values are the standard basis of $\mathbb{R}^4$ denoted by $E$ and corresponding respectively to unoccupied DNA $e_1=(1,0,0,0)$, promoter  occupation $e_2$, and so on). The second line in (\ref{HastySyst}) gives the averaged dynamics for the cro-repressor $cI$ and its dimer $cI_2$. The first reaction describes (reversible) dimerization. The reaction $DcI_2\rightarrow^{k_5}DcI_2+n\ cI$ describes transcription. When the site is occupied in promoter position, $n$ copies of $cI$ are produced. Finally, $cI$ is degraded.
A coherent model (of driver $b\pr{\gamma, x_1,x_2}$) should satisfy the following.
\begin{itemize}
\item[1.] $cI$ and $cI_2$ are averaged (such that the corresponding variables $x_1,x_2$ obey to $x_1^2+x_2^2\leq 1$);
\item[2.] It is constructed around the law of mass action (to cope with the steady-state behavior);
\item[3.] It satisfies bistability :
\item[3i.] Lysis is triggered by a threshold level of $cI$-type repressor (i.e. low level of repressor and repressor dimer, say $x_1^2+x_2^2\leq r$, for some $r>0$ and it is irreversible (i.e. $0$ is a stable attractor).
\item[3ii.] Lysogeny is a stable behavior. Any trajectory starting from a "clear" concentration $x\in K^{lysogeny}\setminus \partial K^{lysogeny}$ (where $K^{lysogeny}:=\set{x\in\mathbb{R}^2:r\leq\norm{x}^2\leq 1}$) stays in $K^{lysogeny}$ until it reaches the threshold level (and lysis is triggered by modifying the coefficient).
\end{itemize}
\subsection{Some Elements of Structure and Border Avoidance Result}
From now on, we consider a piecewise-continuous switched model (a particular case of piecewise  deterministic processes PDMP introduced in \cite{Davis_84, davis_93}) as follows. The process consisting of a couple mode/state is defined on some space $\Omega$ and takes its values in the space $E\times \mathbb{R}^N$. For simplicity (and to avoid some assumptions on infinite activity and uniform tightness of transition measures), $E$ is considered to be a compact subset of some Euclidean space. (For further details on the construction on the Hilbert cube, the reader is referred to \cite{Ikeda_Watanabe_1981}). The family of Borel subsets of $E$ (considered to inherit the topology of the corresponding Euclidean space) will be denoted by $\mathcal{B}(E)$. \\
The process encoding a couple mode/continuous state will be denoted by $\pr{\Gamma,X}$ and described by the characteristic triplet $\pr{b,\theta,Q}$ consisting of:
\begin{itemize}
\item[i.] a bounded, uniformly continuous family $b:E\times\mathbb{R}^N\times U\longrightarrow \mathbb{R}^N$ such that \[ \sup_{\gamma \in E}\norm{b\pr{\gamma,\cdot,\cdot}}_1<\infty.\] By abuse of notation, we will denote this supremum by $\norm{b}_1$.
\item[ii.] a bounded, uniformly continuous jump intensity $\theta:E\times\mathbb{R}^N\times U\longrightarrow \mathbb{R}_+$ such that \[ \norm{\theta}_1:=\sup_{\gamma \in E}\norm{\theta\pr{\gamma,\cdot,\cdot}}_1<\infty.\]
\item[iii.] a transition measure $Q:E\times\longrightarrow\mathcal{P}(E)$, with $\mathcal{P}(E)$ standing for the probability measures on $E$ such that $Q(\gamma,\{\gamma\})=0,\textnormal{ for all }\gamma\in E.$
\end{itemize}
Let us fix the initial mode $\gamma_0\in E$, the initial position $x_0\in \mathbb{R}^N$ and some sequence of measurable controls $u_n\in\mathbb{L}^0\pr{E\times\mathbb{R}^N\times \mathbb{R}_+;U}$. For $(t,\gamma,y)\in\mathbb{R}_+\times E\times\mathbb{R}^N$ and a measurable $v\in\mathbb{L}^0\pr{E\times\mathbb{R}^N\times \mathbb{R}_+;U}$, we define the deterministic flow
\[d\Phi_s^{t,\gamma,y,v}=b\pr{\gamma,\Phi_s^{t,\gamma,y,v},v(\gamma,y,s-t)}ds, \textnormal{ for }s\geq t, \Phi_t^{t,\gamma,y,v}=y.\]
The first jump time $T_1$ has $\theta$ as jump rate i.e.  \[\mathbb{P}\pr{T_1\geq t}=\exp\pr{-\int_0^t\theta\pr{\gamma,\Phi_s^{0,\gamma_0,x_0,u_0},u_0\pr{\gamma_0,x_0,s}}ds}.\]We define $\pr{\Gamma_t^{\gamma_0,x_0,u},X_t^{\gamma_0,x_0,u}}:=\pr{\gamma_0,\Phi_t^{0,\gamma_0,x_0,u_0}}$on $t<T_1$.\\ The post-jump position $\gamma_1:=\Gamma_{T_1}^{\gamma_0,x_0,u}$ has $Q\pr{\gamma_0}$ as distribution conditionally to $\{T_1=\tau\}$ and we set $x_1:=\Phi_{T_1}^{0,\gamma_0,x_0,u_0}$. Next, the inter-jump time is generated according to the conditional distribution
\[\mathbb{P}\pr{T_2-T_1\geq t\mid T_1,\gamma_1,x_1}=\exp\pr{-\int_{T_1}^{T_1+t}\theta\pr{\gamma_1,\Phi_s^{T_1,\gamma_1,x_1,u_1},u_1\pr{\gamma_1,x_1,s-T_1}}ds},\] the process is defined as $\pr{\Gamma_t^{\gamma_0,x_0,u},X_t^{\gamma_0,x_0,u}}:=\pr{\gamma_1,\Phi_t^{T_1,\gamma_1,x_1,u_1}}$on $T_1\leq t<T_2$ while the post-jump position $\gamma_2$ satisfies \[\mathbb{P}\pr{\gamma_2\in A\mid T_2,T_1, \gamma_1,x_1}=Q\pr{\gamma_1,A}\] and so on.
\begin{remark}
\begin{itemize}
\item[i.]For our readers who prefer an operatorial approach, the infinitesimal operator associated to such processes is \[\mathcal{L}^u\varphi(\gamma,x)=\scal{b(x,u),D_x\varphi(\gamma,x)}+\theta(\gamma,x,u)\int_E\pr{\varphi(\gamma',x)-\varphi(\gamma,x)}Q\pr{\gamma,d\gamma'}\]for continuous functions $\varphi$ that are continuously differentiable with respect to $x$ (hence the notation $D_x$ ).
\item[ii.] Further generalization to switched diffusions \begin{equation*}
\begin{split}
\pr{\begin{matrix}
X_t^{\Gamma_0,x,u}\\
\Gamma_t^{\Gamma_0,x,u}
\end{matrix} }
=&\pr{\begin{matrix}
x\\
\Gamma_0
\end{matrix}}+
\int_0^t \pr{\begin{matrix}
b\pr{\Gamma_s^{\Gamma_0,x,u},X_s^{\Gamma_0,x,u},u(s)}\\
0
\end{matrix}}ds+\pr{\begin{matrix}
\sigma\pr{\Gamma_s^{\Gamma_0,x,u},X_s^{\Gamma_0,x,u},u(s)}\\
0
\end{matrix}}dW_s\\
&+\int_0^t\int_E \pr{\begin{matrix}
0\\
z-\Gamma_{s-}^{\Gamma_0,x,u}
\end{matrix}}\mathbbm{1}_{\xi\leq\theta\pr{\Gamma_{s-}^{\Gamma_0,x,u},X_{s-}^{\Gamma_0,x,u},z}}N_\mu\pr{ds,dz,d\xi},\textnormal{ for }t\geq 0.
\end{split}
\end{equation*}
could be treated with similar arguments. In this case, the measurable space $\pr{E,\mathcal{B}(E)}$ is endowed with a finite measure $\mu$ and the system is driven by an ($N$-dimensional) Brownian motion $W$ and an independent Poisson point measure $N_\mu$ defined on $E\times\mathbb{R}_+$ and having as compensator $\hat{N_\mu}\pr{ds,dz,d\xi}=ds\mu(dz)d\xi$. For further details, the reader is referred to \cite{Ikeda_Watanabe_1981}. As in the PDMP case, one should consider (predictable) open-loop control processes of type $\sum_{n\geq 0}u^n\pr{\Gamma_{T_n}^{\Gamma_0,x_0,u}X_{T_n}^{\Gamma_0,x_0,u},s-T_n,}\mathbbm{1}_{T_n<s\leq T_{n+1}}$, where $T_n$ denote the (fictive) jump-times of $N_\mu$.
\item[iii. ] Our results stand without modification if $Q$ is assumed to further depend on $\gamma,x$ provided a further assumption of weak continuity of $Q$ (see \cite[A3]{G8}).
\end{itemize}

\end{remark}
For controlled PDMP, one gets the following simple characterization.
\begin{thm}
\label{ThmMainPDMP}
The set $E\times\mathring K$ is viable with respect to the switched piecewise deterministic system $\pr{\Gamma,X}$ driven by $\pr{b,\theta,Q}$ if and only if, for every $\gamma\in E$ and every $x\in \partial K$, \begin{equation}\label{ViabSwitch}\inf_{u\in U}\scal{b(x,u),\nu_K(x)}\leq 0.\end{equation}
\end{thm}
We will only sketch the proof since it is quasi-identical to the diffusion case. As for the other results, this is postponed to Section \ref{SectionProofs}.
\subsection{A Mathematical Model}
We begin with the description of the mode component $\Gamma$ (using Line 1 in (\ref{HastySyst})). The state space $E=\set{e_1,e_2,e_3,e_4}$ the canonical basis of $R^4$. The jump intensity is computed as the total propensity(-type) function by setting \[\hat{\theta}\pr{\gamma}=k_2\mathbbm{1}_{e_1}\pr{\gamma}+k_{-2}\mathbbm{1}_{e_2}\pr{\gamma}+k_3\mathbbm{1}_{e_1}\pr{\gamma}+k_{-3}\mathbbm{1}_{e_3}\pr{\gamma}+k_4\mathbbm{1}_{e_2}\pr{\gamma}+k_{-4}\mathbbm{1}_{e_4}\pr{\gamma}\] and constructing a regular $\theta\pr{\gamma,x_1,x_2}$ such that \[\begin{split}\theta\pr{\gamma,x_1,x_2}=\hat{\theta}\pr{\gamma},\textnormal{ for }x_1^2+x_2^2\geq 2r\textnormal{ and }\theta\pr{\gamma,x_1,x_2}=0 \textnormal{ for }x_1^2+x_2^2\leq r.\end{split}\]
To construct the post-jump measure (under matrix form), we set $\hat{Q}=\pr{\begin{matrix}
0 &{k_2} & {k_3}& 0 \\
{k_{-2}} & 0 &0 & {k_4}\\
{k_{-3}} & 0 &0 & 0\\
0 & {k_{-4}} & 0 & 0
\end{matrix}}$ and 
$Q\pr{\gamma,\gamma'}:=\frac{\hat{Q}\pr{\gamma,\gamma'}}{{\hat{\theta}\pr{\gamma}}}$, for all $\gamma,\gamma'\in E$. 
\begin{remark}
If one allows the measures $Q$ to depend on $x$, one can base the construction on the actual propensity function
\[\hat{\theta}\pr{\gamma,x_1,x_2}=k_2x_2\mathbbm{1}_{e_1}\pr{\gamma}+k_{-2}\mathbbm{1}_{e_2}\pr{\gamma}+k_3x_2\mathbbm{1}_{e_1}\pr{\gamma}+k_{-3}\mathbbm{1}_{e_3}\pr{\gamma}+k_4\mathbbm{1}_{e_2}\pr{\gamma}+k_{-4}\mathbbm{1}_{e_4}\pr{\gamma}.\]
\end{remark}
To construct the switched flow, we introduce (according to the law of large masses), $b^{steady}:E\times\mathbb{R}^2\longrightarrow\mathbb{R}^2$ given by \[b^{steady}\pr{\gamma,x_1,x_2}=\pr{\begin{matrix} 
-k_1x_1^2-k_6x_1+2k_{-1}x_2+n\mathbbm{1}_{e_2}(\gamma)\\
k_1x_1^2-k_{-1}x_2
\end{matrix}
} \]To guarantee lysogeny, we construct a Lipschitz function $\chi :\mathbb{R}^2\longrightarrow\mathbb{R}_+$ such that $\chi\pr{x_1,x_2}=1$ on $2r\leq x_1^2+x_2^2\leq 1-r$ and $\chi\pr{x_1,x_2}=0$ for $r\geq x_1^2+x_2^2\textnormal{ or }1\leq x_1^2+x_2^2$ . Morover, to cope with lysis, only degradation is considered \[b^{lysis}\pr{\gamma,x_1,x_2}=\pr{\begin{matrix} 
-k_6x_1\\
-k_{-1}x_2
\end{matrix}}.\]
Finally, we set \begin{equation}
\label{bHasty}
b\pr{\gamma,x_1,x_2}=b^{lysis}\pr{\gamma,x_1,x_2}\mathbbm{1}_{x_1^2+x_2^2\leq r}+b^{steady}\pr{\gamma,x_1,x_2}\times \chi(x_1,x_2)\mathbbm{1}_{x_1^2+x_2^2>r}.
\end{equation}
The set of interest is \[K^{lysogeny}:=\set{\pr{x_1,x_2}\in\mathbb{R}^2:r\leq x_1^2+x_2^2\leq 1}.\] 
Owing to the choice of $\chi$ (null on $\partial K^{lysogeny}$) and Theorem \ref{ThmMainPDMP}, whenever the system starts from $x\in \mathring{K}^{lysogeny}$, it never reaches $\partial K^{lysogeny}$. In particular, lysis (triggered by the trajectory reaching $\set{\pr{x_1,x_2}\in\mathbb{R}^2:r= x_1^2+x_2^2}$) can never occur. Thus, no such PDMP model (based solely on averaging and the law of mass action) can accurately account for the bistability of bacteriophage $\lambda$. A further state (cemetery) has to be introduced in $Q$ and jumps to this state should be triggered by external control (environmental cue).
\section{Proofs of the Results}\label{SectionProofs}
\subsection{Proof for Section \ref{SectionNecessary}}
We begin with the proof of the regularity of the function $\zeta$.
\begin{proof}[Proof of Proposition \ref{PropZeta}]
We only prove the continuity property. To this purpose, let us fix $\varepsilon\in (0,\varepsilon_0]$ and some increasing sequence $\pr{\varepsilon_n}_{n\geq 1}\subset (0,\varepsilon_0]$ converging to $\varepsilon$. Let us reason by contradiction and assume that $\zeta$ is not left-continuous at $\varepsilon$. Then, there exists some $\delta>0$ such that $\zeta\pr{\varepsilon_n}>\zeta\pr{\varepsilon}+\delta,$ for all $n\geq 1$. In particular, reasoning for an arbitrary $n\geq 1$, it follows that \[ \zeta\pr{\varepsilon}=\inf_{x\in \mathring{K}, \varepsilon_n\leq\delta_K(x)\leq\varepsilon}\frac{\abs{b^+(x)-b^+\pr{\pi_{\partial K}(x)}}}{\delta_K(x)}. \] Using the compactness of $K$, it follows that  \[ \zeta\pr{\varepsilon}=\frac{\abs{b^+\pr{x^{opt}}-b^+\pr{\pi_{\partial K}(x^{opt})}}}{\varepsilon},\]for some $x^{opt}\in K$ such that $\delta_K\pr{x^{opt}}=\varepsilon$ (thus being away from $\partial K$). Since $x^{opt}\in \mathring{K}$, one exhibits, for $n$ large enough (such that $\pr{\varepsilon-\varepsilon_n}$ be smaller than the radius of some ball centered at $x^{opt}$ and included in $\mathring{K}$), the points $x_n:=x^{opt}-\pr{1-\frac{\varepsilon_n}{\varepsilon}}\pr{x^{opt}-\pi_{\partial K}\pr{x^{opt}}}\in \mathring{K}$ such that \[\delta_K\pr{x_n}\leq\varepsilon_n \textnormal{ and } \lim_{n\rightarrow\infty} x_n=x^{opt}.\]Using the continuity of $x\mapsto \frac{\abs{b^+(x)-b^+\pr{\pi_{\partial K}(x)}}}{\delta_K(x)}$ (away from $\partial K$), one deduces that \[\inf_{n\geq 1}\zeta\pr{\varepsilon_n}\leq \lim_{n\rightarrow\infty}\frac{\abs{b^+\pr{x_n}-b^+\pr{\pi_{\partial K}\pr{x_n}}}}{\delta_K\pr{x_n}}=\frac{\abs{b^+\pr{x^{opt}}-b^+\pr{\pi_{\partial K}\pr{x^{opt}}}}}{\delta_K\pr{x^{opt}}}=\zeta\pr{\varepsilon},\]thus providing us with a contradiction.
\end{proof}
Let us now provide the proof of the necessary local Lipshitz-like criterion stated in Proposition \ref{PropMainNec}
\begin{proof}[Proof of Proposition \ref{PropMainNec}]
First, one notes that the invariance of $\mathring{K}$ (on $[0,T]$) implies $b^+\pr{\pi_{\partial K}(x)}=0,\textnormal{ for }x\in K$ s.t. $\delta_K(x)\leq\varepsilon_0$. Indeed, if one sets $\bar{x}:=\pi_{\partial K}(x)$ and assumes that $b^+(\bar{x})>\delta_0>0$, then, on some neighborhood of $\bar{x}$ (that can be taken of type $\set{y\in K:\norm{y-\bar{x}}<r}$ for some $r>0$) one has $b^+(x)>\delta_0$. By writing down the differential $d\delta_K\pr{X_t^{x}}=-\left\langle \nu_K\pr{\pi_{\partial K}\pr{X_t^{x}}},b\pr{X_t^{x}}\right\rangle dt$, it follows that $X_t^{x}$ reaches $\partial K$ in time $t\leq \frac{\norm{x-\bar{x}}}{\delta_0}$ as long is does not leave the neighborhood of $\bar{x}$ and this can be guaranteed by taking $x$ such that $\norm{x-\bar{x}}\leq\frac{r}{2+2\frac{\norm{b}_0}{\delta_0}}$. This is in contradiction with the invariance of $\mathring{K}$.\\
To prove the bound on the lower limit, we proceed (again) by contradiction. Let us assume that the interior $\mathring{K}$ is invariant, yet $\sup_{\varepsilon>0}\zeta\pr{\varepsilon}=\infty$. Furthermore, we assume (\ref{AssEta}) fails to hold. 
For simplicity reasons, let us denote \[\varepsilon^n:=\sup\left\lbrace \varepsilon\in\left(0,\varepsilon_0\right],\zeta\pr{\varepsilon}\geq max\set{n,n\zeta\pr{\varepsilon_0}}\right\rbrace,\textnormal{ for } n\geq 1.\] 
The reader will easily note that, by left-continuity, $\zeta\pr{\varepsilon^{n}}\geq n\zeta\pr{\varepsilon_0}$, hence $F_\zeta\pr{\frac{1}{\varepsilon^n}}\geq 1-\frac{1}{n}$ and, thus,
 \begin{equation}
\label{ineq0}
\frac{1}{\varepsilon^n}\leq F_\zeta^{-1}\pr{1-\frac{1}{n}}.
\end{equation}
For every $x\in K$ such that $\delta_K(x)\leq\varepsilon^n$, one has $b^+(x)=\abs{b^+(x)-b^+\pr{\pi_{\partial K}(x)}}\geq\zeta\pr{\varepsilon^n}\delta_K(x)\geq n\delta_K(x).$ Then, starting from some $x_n\in \mathring{K}$ such that $\delta_K\pr{x_n}\leq\varepsilon^n$, one has
\begin{equation*}
d\delta_K\pr{X_t^{x_n}}=-\left\langle \nu_K\pr{\bar{X_t^{x_n}}},b\pr{X_t^{x_n}}\right\rangle dt\leq -n\delta_K\pr{X_t^{x_n}}dt.
\end{equation*}
It follows that $x_{n+1}:=X_{t_n}^{x_n}$ satisfies $\delta_K\pr{x_{n+1}}\leq \varepsilon^{n+1}$ for $t_n:=\frac{\ln\varepsilon^n-\ln\varepsilon^{n+1}}{n}$. Arguing by recurrence,  for $k\geq n$, $x_{k}:=X_{t_k}^{x_n}$ satisfies $\delta_K\pr{x_{k}}\leq \varepsilon^{k}$ for $t_k:=\sum_{j=n}^{k}\frac{\ln\varepsilon^j-\ln\varepsilon^{j+1}}{j}.$ In particular, for every \[t>t_\infty:=\sum_{j=n}^{\infty}\frac{\ln\varepsilon^j-\ln\varepsilon^{j+1}}{j},\]one has $X_{t}^{x_n}\in \partial K.$
Owing to (\ref{AssEta}) not holding true, there exists some $\beta>1$ and some constant $C>0$ such that, for $n$ large enough (and using (\ref{ineq0})),
\begin{equation}
\label{BertrandEst}
\frac{-\ln\varepsilon^{k+1}}{k(k+1)}\leq \frac{1}{k\pr{\ln k}^\beta} \frac{1}{k+1}\pr{\ln\pr{k+1}}^\beta\ln \pr{F_\zeta^{-1}\pr{1-\frac{1}{k+1}}}\leq C\frac{1}{k\pr{\ln k}^\beta}.
\end{equation}
Finally, one notes that $t_\infty\leq \sum_{j=n}^{\infty}\frac{-\ln\varepsilon^{j+1}}{j\pr{j+1}}$.
Since, owing to (\ref{BertrandEst}), the last quantity corresponds to the general term of the convergent Bertrand series (with $\alpha=1, \beta>1$), it follows that, for every $T>0$, by picking $n$ large enough, one has $X_t^{x_n}$ exits $\mathring{K}$ prior to $t_\infty<T$. This provides us with a contradiction and the proof is complete.
\end{proof}
\subsection{Krylov's Shaking the Coefficients and Linear Formulations}
Owing to the assertion i. in Proposition \ref{PropKrylov}, one gets the existence of a constant $c_0$ such that \[\mathbb{E}\pp{\int_0^\infty e^{-\lambda t}\norm{X_t^{x,u}}^2dt}\leq c_0\pr{1+\norm{x}^2},\]for all admissible control processes $u$.
We recall some elements taken from \cite{G6}. To any control trajectory $X^{x,u}$ one associates an occupation measure defined by \[\gamma_{x,u}(A)=\lambda\mathbb{E}\pp{\int_{\mathbb{R}_+}e^{-\lambda t}\mathbbm{1}_A\pr{t,X_t^{x,u},u(t)}dt},\]for all Borel sets $A\subset \mathbb{R}_+\times\mathbb{R}\times U$. Whenever convenient, by abuse of notation, the marginal of $\gamma_{x,u}\in\mathcal{P}\pr{\mathbb{R}_+\times\mathbb{R}\times U}$ i.e. $\gamma_{x,u}\pr{\mathbb{R}_+,dy,du}\in \mathcal{P}\pr{\mathbb{R}\times U}$ will still be denoted by $\gamma_{x,u}$. The set of all occupation measures is denoted by $\Gamma(x)\subset \mathcal{P}\pr{\mathbb{R}\times U}$. One shows (cf. \cite[Corollary 2.1]{G6}) that the closed convex hull of $\Gamma(x)$ satisfies\[\Theta(x):=\bar{co}\pr{\Gamma(x)}=\set{\gamma\in \mathcal{P}\pr{\mathbb{R}\times U}:\forall \varphi\in C^2\pr{\mathbb{R}^N}, \int_{\mathbb{R}^N\times U}\pp{\lambda\pr{\varphi(x)-\varphi(x)}+\mathcal{L}^v\varphi(y)}\gamma\pr{dy,dv}=0}. \]\\  This implies that $\Theta(x)\subset \set{\gamma\in \mathcal{P}\pr{\mathbb{R}^N\times U} : \int_{\mathbb{R}^N\times U}\norm{y}^2 \gamma(dy,du)\leq c_0}$ thus being compact.\\ Moreover, since $f_n$ and $\mathbbm{1}_{\pr{\mathring{K}}^c}$ are upper semi-continuous (u.s.c.), one gets \[V(x)=\inf_{\gamma\in \Theta(x)}\int_{\mathbb{R}^N}\mathbbm{1}_{\pr{\mathring{K}}^c}(y)\gamma(dy,U), \ V_n(x)=\inf_{\gamma\in \Theta(x)}\int_{\mathbb{R}^N}f_n(y)\gamma(dy,U),\]for all $x\in \mathbb{R}^N$.
\subsection{Proofs for Section \ref{Section_NearViabDiff}}
First, let us prove the convergence of approximating functions $V_n$ to the viability indicator $V$.
\begin{proof}[Proof of Proposition \ref{VnConvergeToV}]
Let us fix $x \in \mathbb{R}^N$. For every $\gamma \in  \Theta(x)$, the convergence \[ \lim_{n\rightarrow\infty}\int_{\mathbb{R}^N}f_n(y)\gamma(dy,du)=\int_{\mathbb{R}^N}\mathbbm{1}_{\pr{\mathring{K}}^c}(y)\gamma(dy,du)\] follows from point-wise (non-increasing) convergence of $f_n$ to $f$ by applying the dominated convergence theorem. To prove our proposition, we reason by contradiction. To this purpose, let us assume that, for some $\delta>0$, one has $\lim_{n\rightarrow\infty}V_n(x)>V(x)+\delta $. Due to our assumption and the definition of $V_n$, for every $n$ large enough and every $\gamma\in\Theta(x)$, \[\int_{\mathbb{R}^N}f_n(y)\gamma(dy,U)\geq V_n(x)>V(x)+\delta.\]Passing to the limit as $n\rightarrow\infty$ and using the first part of the proof, one gets \[\int_{\mathbb{R}^N}f(y)\gamma(dy,U)\geq V(x)+\delta.\]Taking the infimum over $\gamma\in\Theta(x)$ provides us with a contradiction. The proof is now complete.
\end{proof}
Next, we give the proof of  Lemma \ref{LemmaReductionToKeps} connecting the behavior of $V$ on $\mathring{K}$ to the behavior close to the border.
\begin{proof}[Proof of Lemma \ref{LemmaReductionToKeps}]
We make the notations $K^{\frac{\varepsilon_0}{2}}:=\set{ y\in K\ :\ d_{\partial K}(y)= \frac{\varepsilon_0}{2}}$ and \[\eta_n:=\sup_{y\in K^{\frac{\varepsilon_0}{2}}}V_n\pr{y}.\]
Let us fix, for the time being, $n\geq \frac{2}{\varepsilon}$. In particular, one gets \begin{equation}
\label{fnestim}
f_n(y)=0, \textnormal{ for all } y\in K\setminus K_{\frac{\varepsilon_0}{2}}.
\end{equation} 
We begin with fixing the initial datum $x\in \pr{\mathring{K}\setminus K_\frac{\varepsilon_0}{2}}$ and pick some admissible control process $u^0$. We introduce the stopping time \[\tau:=\inf \left\lbrace t>0 :d_{\partial K}\pr{X_t^{x,u^0}}\leq \frac{\varepsilon_0}{2} \right\rbrace. \]\\
One easily notes, owing to the continuity properties of our solution, that, on the set $\set{\tau<\infty}$, one has $X_{\tau}^{x,u^0}\in K^{\frac{\varepsilon_0}{2}}$ and, thus, $\mathbb{E}\pp{V_n\pr{X_{\tau}^{x,u^0}}\mid X_{0}^{x,u^0}=x}\leq \eta_n$. Then our first assertion follows from this last inequality, the dynamic programming principle and (\ref{fnestim}). Strict inequality is a consequence of the fact that, due to bounded coefficients, $\mathbb{P}\pr{\tau>0}>0$.\\
To prove the second assertion, we begin with proving that 
\begin{equation}
\label{etan}
\lim_{n\rightarrow\infty}\eta_n=0,
\end{equation}whenever the restriction of $V$ to $K_\frac{\varepsilon_0}{2}\setminus\partial K$ is null.
Indeed, if one assumes the contrary, then, for some $\delta>0$ and for every $n\geq 1$ (or, at least, every, $n$ large enough), there exists some $y_n\in K^{\frac{\varepsilon_0}{2}}$ such that $V_n\pr{y_n}\geq \delta.$\\
Since $K^{\frac{\varepsilon_0}{2}}$ is compact, some sub-sequence of $\pr{y_n}_{n\geq 1}$ converges to some $\bar{y}\in K^{\frac{\varepsilon_0}{2}}$. Moreover, due to the monotonicity of $\pr{f_n}_{n\geq 1}$, one has, for every $1\leq n\leq m$,\[V_n(y_m)\geq V_m(y_m)\geq \delta.\] Fixing $n$ and passing to the limit as $m\rightarrow\infty$ (along the sub-sequence mentioned before), it follows that $V_n\pr{\bar{y}}\geq\delta$. Owing to the convergence $\lim_{n\rightarrow\infty}V_n\pr{\bar{y}}=V\pr{\bar{y}}$ (cf. Proposition \ref{VnConvergeToV}), one gets a contradiction. It follows that (\ref{etan}) holds true. 
To conclude the proof of our assertion, one allows $n\rightarrow\infty$ in the first assertion of our proposition and uses (\ref{etan}) and the (point-wise) convergence of $V_n$ to $V$.
\end{proof}
We conclude the subsection with the proof of the main result for Brownian diffusions. 
We consider, on $K_{\varepsilon_0}$, the application $K_{\varepsilon_0}\ni x \mapsto\pi_{\partial K}(x)\in \partial K $ associating to every $x$ its projection on $\partial K$ i.e. the unique $\pi_{\partial K}(x)\in \partial K$ such that $\delta_K(x)=\norm{x-\pi_{\partial K}(x)}$. Moreover, we define the Hamiltonian
\begin{equation}
\label{Hamiltonian}
H\pr{x,r,p,A}=\pr{\lambda-c_\mathcal{L}-c_\sigma^2} r+\sup_{u\in  U: \sigma^*\pr{\pi_{\partial K}(x),u}\nu_K\pr{\pi_{\partial K}(x)}}\pr{-\frac{1}{2}Tr\pp{\sigma \sigma^*\pr{x,u}A}-\left\langle b(x,u),p\right\rangle},
\end{equation}
for all ${x,r,p,A}\in K_{\varepsilon_0}\times\mathbb{R}\times\mathbb{R}^N\times\mathcal{S}^N$.
\begin{proof}
We begin with the proof of the "only if" part. Using the inequality $\mathbbm{1}_{K^c}\leq \mathbbm{1}_{\pr{\mathring{K}}^c}$, the viability indicator function (for the closed set $K$) i.e. \[V_K(y):=\inf_{u\in\mathcal{U}}\mathbb{E}\pp{\int_0^\infty e^{-\lambda t}\mathbbm{1}_{K^c}\pr{X_t^{y,u}}dt},\textnormal{ for all }y\in K\]satisfies $V_K\leq V$. As a consequence, whenever $\mathring{K}$ is near-viable, it follows that $V_K(x)=0$, for all $x\in \mathring{K}$. Standards arguments yield the lower semi-continuity of $V_K$ and, owing to the regularity assumptions guaranteeing $K=\bar{\mathring{K}}$, it follows that $V_K\pr{\bar{x}}=0$, for all $\bar{x}\in \partial K$.\\
Let us now turn to the proof of the "if" part. \\
\underline{Step 1}. (The elements of this step follow closely Krylov's method in \cite{Krylov_00} adapted to our elliptic framework).
We begin with defining, for $\delta>0$, the $\delta$-shaken coefficients value function(s)\[V_{n,\delta}(x):=\inf_{(u,e)\in \mathcal{U}\times\mathcal{E}}\mathbb{E}\pp{\int_0^\infty e^{-\lambda t} f_n\pr{X_t^{x,u,e,\delta}}dt},\ x\in \mathbb{R}^N.\]Using the $n$-Lipschitz properties of $f_n$ together with Proposition \ref{PropKrylov} (assertions ii. and iii.) and owing to the choice of $\lambda\geq 2\lambda_0$, one deduces that $V_n$, $V_{n,\delta}$ are $n$-Lipschitz  and \[\norm{V_{n,\delta}-V_n}_\infty\leq n\delta.\]
By considering a sequence of standard mollifiers $\varphi_\delta$ and owing to the $n$-Lipschitz property of $V_n$, one gets that the convoluted function $V_n^\delta:=V_{n,\delta}\ast\varphi_\delta$ satisfies \[\abs{V_n^\delta(x)-V_n(x)}\leq n\delta,\textnormal{ for all }x\in\mathbb{R}^N.\]
Moreover, $V_n^\delta$  will be a regular subsolution to the Hamilton-Jacobi-Bellman equation
\[\lambda\varphi+\sup_{u\in U}\pr{-\mathcal{L}^u\varphi(x)}-f_n(x)=0,\]on $\mathbb{R}^N$ i.e., for every $\pr{x,u}\in\mathbb{R}^N\times U$, one has
\begin{equation}
\label{VnDeltaSupersol}
-\lambda V_n^\delta(x)+\frac{1}{2}Tr\pp{\sigma\sigma^*\pr{x,u}D^2V_n^\delta(x)}+\left\langle b\pr{x,u},DV_n^\delta(x)\right\rangle+f_n(x)\geq 0.
\end{equation}
\underline{Step 2}. 
The reader is invited to recall the construction of the function $g\in C^{2,1}\pr{\mathbb{R}^N}$ given in Proposition \ref{PropToolsK}. If the closed set $K$ is near-viable, then, according to [BCQ, Theorem A.1. iii], for every $x\in \partial K$ one has 
\begin{equation}
\label{BCQdeltaK}
\sup_{u\in U\ \textnormal{s.t.}\ \sigma^*(x,u)\nu_K(x)=0}\pr{\mathcal{L}^u\pr{\delta_K(x)}}\geq 0.
\end{equation}
Let us now consider, for $n\geq 1$ and $\delta>0$ (arbitrary for the time being), the function \[W_n^\delta:=-gV_n^\delta.\] Then $W_n^\delta$ belongs to $C^{2,1}\pr{\mathbb{R^N}}$. We claim that, whenever $x\in \pr{K_{\varepsilon_0}\setminus K^{\varepsilon_0}}\cap \textnormal{Argmax}_{loc,K} \pr{-W_n^\delta}$ is a local maximum with respect to $K$,  the function $W_n^\delta$ satisfies (at $x$),
\[H\pr{x,\varphi,D\varphi,D^2\varphi}+f_n(x)\delta_K(x)\geq 0,\]where the Hamiltonian is given by 
(\ref{Hamiltonian}).  \\
Case 1. We begin with proving this assertion whenever $x\in \partial{K}$. At such points, $W_n^\delta(x)=f_n(x)g(x)=0.$ Moreover, $DW_n^\delta(x)=-V_n^\delta(x)D\delta_K(x)=V_n^\delta(x)\nu_K(x)\textnormal{ and } D^2W_n^\delta(x)=2DV_n^\delta(x)\otimes \nu_K(x)-V_n^\delta(x)D^2\delta_K(x).$ Here, $\otimes$ denotes the usual tensor product on $\mathbb{R}^N$. Then, using (\ref{BCQdeltaK}), one gets 
\begin{equation}
\begin{split}
&H\pr{x,0,DW_n^\delta(x),D^2W_n^\delta(x)}\\
=&\sup_{\sigma^*\pr{x,u}\nu_K(x)=0}\pr{-\frac{1}{2}Tr\pp{\sigma \sigma^*\pr{x,u}\pr{2DV_n^\delta(x)\otimes \nu_K(x)-V_n^\delta(x)D^2\delta_K(x)}}-\left\langle b(x,u),V_n^\delta(x)\nu_K(x)\right\rangle} \\
=&V_n^\delta(x)\sup_{\sigma^*\pr{x,u}\nu_K(x)=0}\pr{\frac{1}{2}Tr\pp{\sigma \sigma^*\pr{x,u}D^2\delta_K(x)}+\left\langle b(x,u),D\delta_K(x)\right\rangle}\geq 0.
\end{split}
\end{equation}
Case 2. Let us now focus on $x\in K_{\varepsilon_0}\setminus \pr{\partial K\cup K^{\varepsilon_0}}$. In this case, using the classical maximum principle (and recalling that $W_n^\delta$ belongs to $C^{2,1}\pr{\mathbb{R}}^N$), one has
\[
\begin{split}
0=DW_n^\delta(x)&=-V_n^\delta(x)D\delta_K\pr{x}-\delta_K(x)DV_n^\delta(x)=V_n^\delta(x)\nu_K\pr{\pi_{\partial K}(x)}-\delta_K(x)DV_n^\delta(x),\\
0\leq D^2W_n^\delta(x)&=-V_n^\delta(x)D^2\delta_K(x)+2DV_n^\delta(x)\otimes \nu_K\pr{\pi_{\partial K}(x)}-\delta_K(x)D^2V_n^\delta(x).
\end{split}
\]Therefore, by computing the Hamiltonian one has
\[
\begin{split}
&\pr{\lambda-c_\mathcal{L}-c_\sigma^2}W_n^\delta(x)+f_n(x)\delta_K(x)\\
\geq&H\pr{x,W_n^\delta(x),DW_n^\delta(x),D^2W_n^\delta(x)}+f_n(x)\delta_K(x)\\
=&\sup_{\sigma^*\pr{\pi_{\partial K}(x),u}\nu_K\pr{\pi_{\partial K}(x)}=0}\pr{
\begin{split}
&V_n^\delta(x)\pp{c_\mathcal{L}\delta_K(x)+\frac{1}{2}Tr\pp{\sigma \sigma^*\pr{x,u}D^2\delta_K(x)}+\left\langle b(x,u),D\delta_K(x)\right\rangle}\\
+&\delta_K(x)\pp{-\lambda V_n^\delta(x)+\frac{1}{2}Tr\pp{\sigma \sigma^*\pr{x,u}D^2V_n^\delta(x)}+\left\langle b(x,u),DV_n^\delta(x)\right\rangle+f_n(x)}\\
-&c_\sigma^2 W_n^\delta(x)-Tr\pp{\sigma \sigma^*\pr{x,u}DV_n^\delta(x)\otimes \nu_K\pr{\pi_{\partial K}(x)}}\\
\end{split}} \\
=&H\pr{x,W_n^\delta(x),DW_n^\delta(x),D^2W_n^\delta(x)}+f_n(x)\delta_K(x)\\
=&\sup_{\sigma^*\pr{\pi_{\partial K}(x),u}\nu_K\pr{\pi_{\partial K}(x)}=0}\pr{
\begin{split}
&V_n^\delta(x)\pp{c_\mathcal{L}\delta_K(x)+\mathcal{L}^u\delta_K(x)}\\
+&\delta_K(x)\pp{-\lambda V_n^\delta(x)+\frac{1}{2}Tr\pp{\sigma \sigma^*\pr{x,u}D^2V_n^\delta(x)}+\left\langle b(x,u),DV_n^\delta(x)\right\rangle+f_n(x)}\\
-&c_\sigma^2 W_n^\delta(x)-\frac{V_n^\delta(x)}{\delta_K(x)}\norm{\sigma^*\pr{x,u}\nu_K\pr{\pi_{\partial K}(x)}}^2\\
\end{split}} \\
\end{split}
\]
Recalling the notations (\ref{csigmacL}) (see also (\ref{C0})), it follows from this last inequality that 
\[
\begin{split}
&\pr{\lambda-c_\mathcal{L}-c_\sigma^2}W_n^\delta(x)+f_n(x)\delta_K(x)\\
\geq &\sup_{\sigma^*\pr{\pi_{\partial K}(x),u}\nu_K\pr{\pi_{\partial K}(x)}=0}\pr{
\begin{split}
&V_n^\delta(x)\mathcal{L}^u\delta_K\pr{\pi_{\partial K}(x)}\\
+&\delta_K(x)\pp{-\lambda V_n^\delta(x)+\frac{1}{2}Tr\pp{\sigma \sigma^*\pr{x,u}D^2V_n^\delta(x)}+\left\langle b(x,u),DV_n^\delta(x)\right\rangle+f_n(x)}\\
-&c_\sigma^2W_n^\delta(x)-c_\sigma^2\delta_K(x)V_n^\delta(x)
\end{split}}
\end{split}
\]
Owing to (\ref{VnDeltaSupersol}) and (\ref{BCQdeltaK}), one deduces that 
\begin{equation}\label{Ineq1}
\pr{\lambda-c_\mathcal{L}-c_\sigma^2}W_n^\delta(x)+f_n(x)\delta_K(x)\geq 0.
\end{equation}
\underline{Step 3}. In order to prove the near-viability of the open set $\mathring{K}$, we  proceed by contradiction. Let us assume the existence of some $\alpha>0$ and some $x_\alpha\in K_\frac{\varepsilon_0}{2}$ such that $\delta_K(x_\alpha)V\pr{x_\alpha}\geq 2\alpha$. Recalling that $V_n\geq V$ (in fact, point-wise convergence of $V_n$ to $V$ suffices) and $\abs{V_n^{n^{-2}}(x)-V_n(x)}\leq \frac{1}{n}$ (for $x\in \mathbb{R}^N$, see Step 1), it follows that $\max_{x\in K_{\varepsilon_0}}\pr{-W_n^{n^{-2}}(x)}\geq \alpha$ for all $n$ large enough. Let us now denote by $x_n\in K$ the global maximum of $-W_n^{n^{-2}}$ with respect to $K$.\\ 
\underline{Step 3.1}. Owing to Lemma \ref{LemmaReductionToKeps} (and to the fact that $W_n^{n^{-2}}\pr{x}=0$ as soon as $x\in \partial K$), it follows that $x_n\in K_{\varepsilon_0}\setminus \pr{\partial K\cup K^{\varepsilon_0}}$ (at least for $n$ large enough).\\ Indeed, for points $y\in K\setminus K_{\varepsilon_0}\cup K^{\varepsilon_0}$, one has, using the estimates in Step 1 and the dynamic programming principle (with $\tau^u$, as before, the time of $X^{y,u}$ hitting $K_{\frac{\varepsilon_0}{2}}$ with admissible control $u$),
\begin{equation}
\label{goodoptim}
\begin{split}
\delta_K(y)V_n^{\frac{1}{n^2}}(y)&\leq \frac{diam(K)}{n}+\delta_K(y)V_n(y)\leq \frac{diam(K)}{n}+\delta_K(y)\inf_{u\in \mathcal{U}}\mathbb{E}\pp{e^{-\lambda\tau^u}V_n\pr{X_{\tau^u}^{y,u}}}\\
&\leq \frac{diam(K)}{n}+\frac{2}{\varepsilon_0}\delta_K(y)\pp{\mathbb{P}\pr{\tau^u<t^*}+e^{-\lambda t^*}}\sup_{y'\in K^{\frac{\varepsilon_0}{2}}}\delta_K(y')V_n\pr{y'}.
\end{split}
\end{equation}
One recalls that $t^*$ has been introduced in (\ref{Ineqlambda1}) and notes that, since $y$ is at least $\frac{\varepsilon_0}{2}$-far from $K^{\frac{\varepsilon_0}{2}}$, 
\begin{equation*}
\mathbb{P}\pr{\tau^u<t^*}\leq \mathbb{P}\pr{\sup_{0\leq t\leq t^*}\norm{X_t^{y,u}-y}\geq \frac{\varepsilon_0}{2}}\leq \frac{\varepsilon_0}{8diam(K)}.
\end{equation*}
Using the choice of $\lambda$ (again by (\ref{Ineqlambda1})) in (\ref{goodoptim}), we get
\begin{equation}
\begin{split}
\delta_K(y)V_n^{\frac{1}{n^2}}(y)\leq \frac{diam(K)}{n}+\frac{1}{2}\sup_{y'\in K^{\frac{\varepsilon_0}{2}}}\delta_K(y')V_n\pr{y'}<\frac{2diam(K)}{n}+\frac{1}{2}\sup_{y'\in K^{\frac{\varepsilon_0}{2}}}\delta_K(y')V_n^{\frac{1}{n^2}}\pr{y'}.
\end{split}
\end{equation}
For $n$ large enough (s.t. $\frac{2diam(K)}{n}<\frac{\alpha}{2}$), it follows, as announced, that $x_n\in K_{\varepsilon_0}\setminus \pr{\partial K\cup K^{\varepsilon_0}}$.\\
\underline{Step 3.2}. Since $x_n$ is a global maximum, using Step 2, one gets 
\[\frac{1}{4n}=\max_{x\in K_\frac{\varepsilon_0}{2}}f_n(x)\delta_K(x)\geq f_n(x_n)\delta_K(x_n)\geq -\pr{\lambda-c_\mathcal{L}-c_\sigma^2}W_n^{n^{-2}}(x)\geq \pr{\lambda-c_\mathcal{L}-c_\sigma^2}\alpha.\]
Letting $n\rightarrow \infty$ leads to a contradiction. 
\end{proof}
\subsection{Sketch of the Proof(s) for Section \ref{SectionPDMP}}
\begin{proof}[(Sketch of the) Proof of Theorem \ref{ThmMainPDMP}]
To prove the necessity of the condition (\ref{ViabSwitch}), one simply notes that near-viability of $E\times\mathring K$ implies the same property for $E\times K$ (as it was the case for diffusions). Finally, the condition (\ref{ViabSwitch}) follows from \cite{G8}.\\
Let us now give some elements for the sufficiency. Lemma \ref{LemmaReductionToKeps} needs no modification (if not quoting the dynamic programming principle in \cite{Soner86_2}; see also \cite{G8}). Concerning the modifications to the proof of Theorem \ref{ThmMainDiff} one proceeds as follows. The estimates in Step 1 are no longer of type $n\delta$ but given by some modulus of continuity taken from \cite[Section 6.1]{Goreac_SIAM_2015} (see also \cite[Theorem 3.6]{G8}). Step 2 does not need any changes. Indeed, among the three terms on the right-side in Case 2 (the only one of real interest) one notes the following. The first term applies to $\delta_K$ and the integral term in the infinitesimal operator is $0$. The second term is lower-bounded by $0$ as consequence of the construction of the regular subsolution. Finally, all the terms in which $\sigma$ appears are null.  For Step 3, one no longer picks $\delta=\frac{1}{n^2}$, but some $\delta_n$ (according to Step 1), such that $\abs{V_n^{\delta_n}\pr{\gamma_\alpha,x_\alpha}-V_n\pr{\gamma_\alpha,x_\alpha)}}\leq \frac{1}{n}$. (In this case, the reasoning by contradiction would provide some $\pr{\gamma_\alpha,x_\alpha}$ for which $\delta_K\pr{x_\alpha}V\pr{\gamma_\alpha,x_\alpha}\geq 2\alpha$).
\end{proof}
\bibliographystyle{plain}
\bibliography{bibliografie_2018}
\end{document}